\documentclass[11pt,amssymb,amsfont,a4paper]{article}
\usepackage{latexsym,amssymb,amsmath,amsthm,color}
\usepackage{geometry}
\usepackage{url}
\usepackage{graphicx}
\usepackage[utf8]{inputenc}
\usepackage{authblk}
\usepackage{multirow}

\theoremstyle{definition}
\newtheorem{definition}{Definition}

\newtheorem{example}[definition]{Example}

\theoremstyle{plain}
\newtheorem{theorem}{Theorem}
\newtheorem{proposition}[definition]{Proposition}
\newtheorem{lemma}[definition]{Lemma}
\newtheorem{remark}[definition]{Remark}
\newtheorem{corollary}[definition]{Corollary}

\title{Zeros with multiplicity, Hasse derivatives and\\linear factors of general skew polynomials}
\author{Umberto Mart{\'i}nez-Pe\~{n}as \thanks{umberto.martinez@uva.es}}
\affil{IMUVa-Mathematics Research Institute,\\University of Valladolid, Spain}

\date{}

\begin{document}

\maketitle

\begin{abstract}
In this work, multiplicities of zeros of skew polynomials are studied. Two distinct definitions are considered: First, $ a $ is said to be a zero of $ F $ of multiplicity $ r $ if $ (x-a)^r $ divides $ F $ on the right; second, $ a $ is said to be a zero of $ F $ of multiplicity $ r $ if some skew polynomial $ P = (x-a_r) \cdots (x-a_2) (x-a_1) $, having $ a_1 = a $ as its only right zero, divides $ F $ on the right. Neither of these two notions implies the other for general skew polynomials. We show that, in the first case, Lam and Leroy's concept of P-independence does not behave naturally, whereas a union theorem still holds. In contrast, we show that P-independence behaves naturally for the second notion of multiplicities. As a consequence, we provide extensions of classical commutative results to general skew polynomials. These include: (1) The upper bound on the number of (P-independent) zeros (counting multiplicities) of a skew polynomial by its degree, and (2) The equivalence of P-independence, Hermite interpolation and the invertibility of confluent Vandermonde matrices (for which we introduce skew polynomial Hasse derivatives). 



\end{abstract}

\section{Introduction} \label{sec intro}

Multiplicities of zeros of polynomials constitute a central elementary concept with a wide range of generalizations and applications. One of the most basic results states that, if $ F $ is a non-zero polynomial over a field and $ a_1, a_2, \ldots, a_n $ are all of its $ n $ distinct roots (in the same field), with multiplicities $ r_1, r_2, \ldots, r_n $, respectively, then
\begin{equation}
\sum_{i=1}^n r_i \leq \deg(F),
\label{eq degree bound intro}
\end{equation}
where equality holds for all $ F $ if, and only if, the field is algebraically closed. The bound (\ref{eq degree bound intro}) is equivalent to Hermite interpolation (in the sense of Theorems \ref{th hermite interpolation nonconj naive case} and \ref{th hermite interpolation II} using Hasse derivatives \cite{hasse} in positive characteristic) and the invertibility of confluent Vandermonde matrices \cite{kalman}.

However, multiplicities in non-commutative algebra are not understood to the same extent. The study of zeros of conventional polynomials over division rings was initiated in \cite{wedderburn} and later \cite{gordon}. Zeros of quaternionic regular functions (which include polynomials) were studied in \cite{gentili}, and their multiplicities were considered in \cite[Def. 5.5]{gentili-zero}. 

In an alternative line of research, Lam and Leroy \cite{lam, lam-leroy} described exactly the structure of zeros (without counting multiplicities) of general skew polynomials \cite{ore}, which include conventional polynomials over division rings \cite{wedderburn}, linearized polynomials over finite fields \cite{orespecial}, and differential polynomials \cite{jacobson-derivation}, among others. Lam and Leroy's works characterize evaluation points where (\ref{eq degree bound intro}) holds without counting multiplicities. Such sets of points are called \textit{P-independent}. Lam and Leroy's results admit natural generalizations to (free) multivariate skew polynomial rings \cite{lin-multivariateskew, multivariateskew}. 

In \cite{eric, gentili-zero, johnson-thesis}, an element $ a $ is considered to be a zero (say, on the right) of a skew polynomial $ F $ with multiplicity $ r \in \mathbb{N} $ if $ (x-a)^r $ divides $ F $ on the right. Although this notion works fine in the quaternionic case \cite{bolotnikov-conf}, we will show that it becomes quite pathological for general skew polynomials (see Examples \ref{ex power of linear term is no multiplicity} and \ref{ex full conjugacy class pathology for naive mult}).

An interesting alternative notion of multiplicity for zeros of quaternionic polynomials was recently given by Bolotnikov \cite{bolotnikov-conf, bolotnikov-zeros}. Bolotnikov defines $ a $ as a zero of $ F $ of multiplicity $ r \in \mathbb{N} $ if there exists a skew polynomial of degree $ r $ that factors completely 
\begin{equation}
P = (x-a_r) \cdots (x - a_2) (x-a_1)
\label{eq P intro}
\end{equation} 
such that $ P $ divides $ F $ on the right, $ a = a_1 $ and $ a_1 $ is the only zero of $ P $ (on the right). Over the quaternions, this latter condition on $ P $ holds if $ a_1 = a_2 = \ldots = a_r $ (see Corollary \ref{cor multiplicity over quaternions}), but in contrast with the commutative case, the converse does not hold. In other words, $ a_1 $ may be the only right zero of $ P $, yet the set $ \{ a_1, a_2, \ldots, a_r \} $ may have more than one element. As we will show, even the direct implication does not hold for general skew polynomials. In other words, $ (x-a)^r $ may have more than one zero without counting multiplicities (see Example \ref{ex power of linear term is no multiplicity}).

In this manuscript, we study separately both definitions of multiplicity of zeros of general skew polynomials. We extend classical results, including the degree bound (\ref{eq degree bound intro}) above (Theorems \ref{th upper bound multiplicities by degree general}, \ref{th deg bound for nonconj naive case} and \ref{th upper bound multiplicities by degree II}), Hermite interpolation based on Hasse derivatives or confluent Vandermonde matrices (Theorems \ref{th lagrange interpolation}, \ref{th hermite interpolation nonconj naive case} and \ref{th hermite interpolation II}), or multiplicity enhancements of Hilbert's Theorem 90 (Theorem \ref{th min skew pol of conjugacy class with mult} and Corollary \ref{cor hilbert 90}). We will also characterize skew polynomials that have a unique factorization in linear terms (Theorem \ref{th charact multiplicity seqs in alg conj}).

\section{High-degree evaluation points} \label{sec high-degree zeros}

In this section, we provide an abstract framework for generalizing the study of zeros of skew polynomials by Lam and Leroy \cite{lam, lam-leroy}. The final objective is to be able to handle multiplicities of classical zeros of skew polynomials (Sections \ref{sec multi from powers} and \ref{sec multi from linear terms}). 

Skew polynomial rings were introduced by Ore \cite{ore} as follows. Fix a division ring $ \mathbb{F} $, let $ \mathbb{N} $ be the set of natural numbers including $ 0 $, and set $ \mathbb{Z}_+ = \mathbb{N} \setminus \{ 0 \} $. Let $ \mathcal{R} $ be the left vector space over $ \mathbb{F} $ with basis $ \{ x^i \mid i \in \mathbb{N} \} $, where we denote $ 1 = x^0 $ and $ x = x^1 $. The degree of a non-zero $ F = \sum_{i \in \mathbb{N}} F_i x^i \in \mathcal{R} $, where $ F_i \in \mathbb{F} $ for all $ i \in \mathbb{N} $, is the maximum $ i $ such that $ F_i \neq 0 $, and is denoted by $ \deg(F) $. We define $ \deg(F) = - \infty $ if $ F = 0 $.

A product in $ \mathcal{R} $ turns it into a (non-commutative) ring with identity $ 1 $, with $ \mathbb{F} \subseteq \mathcal{R} $ as a subring, where $ x^i x^j = x^{i+j} $, for all $ i,j \in \mathbb{N} $, and $ \deg(FG) = \deg(F) + \deg(G) $ for all $ F,G \in \mathcal{R} $, if, and only if, there exist $ \sigma, \delta : \mathbb{F} \longrightarrow \mathbb{F} $ such that
\begin{equation}
xa = \sigma(a) x + \delta(a),
\label{eq product over constants and variables}
\end{equation}
for all $ a \in \mathbb{F} $, where $ \sigma : \mathbb{F} \longrightarrow \mathbb{F} $ is a ring endomorphism and $ \delta : \mathbb{F} \longrightarrow \mathbb{F} $ is a \textit{$ \sigma $-derivation}: That is, $ \delta $ is additive and $ \delta(ab) = \sigma(a)\delta(b) + \delta(a)b $, for all $ a,b \in \mathbb{F} $.

For each such pair $ (\sigma,\delta) $, we denote $ \mathcal{R} = \mathbb{F}[x; \sigma, \delta] $ when the product is given by (\ref{eq product over constants and variables}), and we call $ \mathbb{F}[x; \sigma, \delta] $ the \textit{skew polynomial ring} over $ \mathbb{F} $ with morphism $ \sigma $ and derivation $ \delta $. The \textit{conventional polynomial ring} $ \mathbb{F}[x] $ is obtained by choosing the identity morphism $ \sigma = {\rm Id} $ and the zero derivation $ \delta = 0 $. Moreover, $ \mathbb{F}[x; \sigma, \delta] $ is commutative if, and only if, $ \mathbb{F} $ is a field, $ \sigma = {\rm Id} $ and $ \delta = 0 $. This will be called \textit{the commutative case}.

The ring $ \mathbb{F}[x; \sigma, \delta] $ is a right Euclidean domain \cite{ore}. Thanks to this property, Lam and Leroy \cite{lam, lam-leroy} provided a natural definition of evaluation by forcing a ``Remainder Theorem'' on the right. We now extend this notion of evaluation as follows.

\begin{definition} [\textbf{High-degree evaluation}] \label{def evaluation high degree}
Given $ F,P \in \mathbb{F}[x;\sigma,\delta] $, we define the evaluation of $ F $ at $ P $, denoted by $ F(P) $, as the remainder of $ F $ when divided by $ P $ by Euclidean division on the right. Thus $ F(P) \in \mathbb{F}[x;\sigma, \delta] $ is a skew polynomial such that $ \deg(F(P)) < \deg(P) $. We say that $ P $ is a zero of $ F $ if $ F(P) = 0 $.
\end{definition}

The evaluation of $ F \in \mathbb{F}[x;\sigma,\delta] $ in $ a \in \mathbb{F} $ given in \cite{lam, lam-leroy} corresponds to the evaluation of $ F $ in $ x - a \in \mathbb{F}[x;\sigma,\delta] $ according to Definition \ref{def evaluation high degree}, and will be denoted $ F(a) \in \mathbb{F} $.

In this manuscript, we will consider non-empty sets $ \mathcal{U} \subseteq \mathbb{F}[x;\sigma,\delta] $ such that $ \deg(F) \geq 1 $, for all $ F \in \mathcal{U} $. The set $ \mathcal{U} $ indicates where we may evaluate, thus where we may take zeros from. The case of classical zeros of skew polynomials is recovered by considering 
\begin{equation}
\mathcal{U}_0 = \{ x-a \in \mathbb{F}[x;\sigma,\delta] \mid a \in \mathbb{F} \}.
\label{eq universal set 0}
\end{equation}
We will identify a set $ \Omega \subseteq \mathcal{U}_0 $ with its underlying set $ \{ a \in \mathbb{F} \mid x-a \in \Omega \} \subseteq \mathbb{F} $. In Section \ref{sec multi from powers}, we will count multiplicities by considering zeros in the universal set
\begin{equation}
\mathcal{U}_1 = \left\lbrace (x-a)^r \in \mathbb{F}[x;\sigma,\delta] \mid a \in \mathbb{F}, r \in \mathbb{Z}_+ \right\rbrace ,
\label{eq universal set I}
\end{equation}
and from Section \ref{sec multi from linear terms} on, we will consider the universal set
\begin{equation}
\mathcal{U}_2 = \left\lbrace P_\mathbf{a} \in \mathbb{F}[x;\sigma,\delta] \mid Z(P_\mathbf{a}) = \{ a_1 \}, \mathbf{a} \in \mathbb{F}^r, r \in \mathbb{Z}_+ \right\rbrace ,
\label{eq universal set II}
\end{equation}
where $ P_\mathbf{a} = (x-a_r) \cdots (x-a_2)(x-a_1) \in \mathbb{F}[x;\sigma,\delta] $ if $ \mathbf{a} = (a_1, a_2, \ldots, a_r) \in \mathbb{F}^r $. Whereas it trivially holds that $ \mathcal{U}_0 \subseteq \mathcal{U}_1 $ and $ \mathcal{U}_0 \subseteq \mathcal{U}_2 $, neither the inclusion $ \mathcal{U}_1 \subseteq \mathcal{U}_2 $ nor the inclusion $ \mathcal{U}_2 \subseteq \mathcal{U}_1 $ hold in general (see Example \ref{ex power of linear term is no multiplicity} and Section \ref{sec particular cases}). 
%

The following definition is inspired by classical Algebraic Geometry.

\begin{definition} [\textbf{High-degree zeros}]
Given $ A \subseteq \mathbb{F}[x;\sigma,\delta] $, its zero set in $ \mathcal{U} $ is
$$ Z_\mathcal{U}(A) = \{ P \in \mathcal{U} \mid F(P) = 0, \textrm{ for all } F \in A \}. $$
Similarly, given a set $ \Omega \subseteq \mathcal{U} $, the ideal of skew polynomials vanishing at $ \Omega $ is 
$$ I(\Omega) = \{ F \in \mathbb{F}[x;\sigma,\delta] \mid F(P) = 0, \textrm{ for all } P \in \Omega \}. $$
\end{definition}

Some properties of zero sets and ideals will require the following generalization of set inclusion (we also use it in Corollary \ref{cor monotonicity}, Proposition \ref{prop monotonicity multiplicity I} and Theorem \ref{th P-independence multiplicity seqs}).

\begin{definition} \label{def subseteq leq}
Given sets $ \Psi, \Omega \subseteq \mathbb{F}[x;\sigma, \delta] $, we say that $ \Psi \leq \Omega $ if, for every finite non-empty subset $ \{ Q_1, Q_2, \ldots, $ $ Q_n \} $ $ \subseteq \Psi $ of size $ n $, there exists a subset $ \{ P_1, P_2, \ldots, P_n \} \subseteq \Omega $, also of size $ n $, such that $ Q_i $ divides $ P_i $ on the right, for $ i = 1,2, \ldots, n $.
\end{definition}
%

\begin{example}
Clearly $ \Psi \leq \Omega $ whenever $ \Psi \subseteq \Omega $. However, $ \Psi \leq \Omega $ holds in many other cases. For example, if $ \Psi = \{ (x-a)^r \} $, for $ a \in \mathbb{F} $ and $ r \in \mathbb{Z}_+ $, then $ \Psi \leq \Omega $ for any $ \Omega = \{ (x-a)^s \} $, where $ s \geq r $. The interpretation, in this case, is that $ \Psi $ and $ \Omega $ correspond to the point $ a $ with multiplicities $ r $ and $ s $, and $ \Psi \leq \Omega $ indicates that the same holds for their multiplicities, i.e., $ r \leq s $. More generally, $ \Psi \leq \Omega $ for $ \Psi = \{ (x-a_1)^{r_1}, \ldots, (x-a_n)^{r_n} \} $ and $ \Omega = \{ (x-a_1)^{s_1}, \ldots, (x-a_n)^{s_n} \} $, where $ a_1, a_2, \ldots, a_n \in \mathbb{F} $ and $ r_i ,s_i \in \mathbb{Z}_+ $ are such that $ r_i \leq s_i $, for $ i = 1,2, \ldots, n $, and $ \Omega $ has size $ n $. Note that, in these cases, $ \Psi \subseteq \Omega $ does not hold unless $ r_i = s_i $ for all $ i = 1,2, \ldots, n $.
\end{example}

Throughout the manuscript, given $ F \in \mathbb{F}[x;\sigma,\delta] $, we denote by $ (F) = \{ GF \mid G \in \mathbb{F}[x;\sigma,\delta] \} $ the left ideal in $ \mathbb{F}[x;\sigma,\delta] $ generated by $ F $. An important feature of this abstract framework is that the following basic properties still hold as in classical Algebraic Geometry.

\begin{proposition} \label{prop properties alg geom}
Let $ \Omega, \Omega_1, \Omega_2 \subseteq \mathcal{U} $ and $ A, A_1, A_2 \subseteq \mathbb{F}[x;\sigma,\delta] $ be arbitrary sets. 
\begin{enumerate}
\item
The set $ I(\Omega) $ is a left ideal of $ \mathbb{F}[x;\sigma,\delta] $.
\item
Given $ P \in \mathcal{U} $, it holds that $ I(\{ P \}) = (P) $ and $ P \in Z_\mathcal{U}(\{ P \}) $.
\item 
Assume that if $ P \in \mathcal{U} $, $ Q \in \mathbb{F}[x;\sigma,\delta] $ divides $ P $ on the right and $ \deg(Q) \geq 1 $, then $ Q \in \mathcal{U} $. Then $ Z_\mathcal{U}(\{ P \}) = \{ P \} $ if, and only if, $ P $ is irreducible.
\item
$ I(\varnothing) = (1) $ and $ Z_\mathcal{U}(\{ 1 \}) = \varnothing $.
\item
If $ \Omega_1 \leq \Omega_2 $, then $ I(\Omega_2) \subseteq I(\Omega_1) $.
\item
If $ A_1 \subseteq A_2 $, then $ Z_\mathcal{U}(A_2) \subseteq Z_\mathcal{U}(A_1) $.
\item
$ I(\Omega_1 \cup \Omega_2) = I(\Omega_1) \cap I(\Omega_2) $.
\item
$ Z_\mathcal{U}(A) = Z_\mathcal{U}((A)) $ and $ Z_\mathcal{U}(A_1 \cup A_2) = Z_\mathcal{U}((A_1) + (A_2)) = Z_\mathcal{U}(A_1) \cap Z_\mathcal{U}(A_2) $.
\item
$ \Omega \subseteq Z_\mathcal{U}(I(\Omega)) $ and equality holds if, and only if, $ \Omega = Z_\mathcal{U}(B) $ for some $ B \subseteq \mathbb{F}[x;\sigma,\delta] $.
\item
$ A \subseteq (A) \subseteq I(Z_\mathcal{U}(A)) $ and equality holds if, and only if, $ A = I(\Psi) $ for some $ \Psi \subseteq \mathcal{U} $.
\end{enumerate}
\end{proposition}
\begin{proof}
We prove the following items:
\begin{enumerate}
\item
If $ F \in \mathbb{F}[x;\sigma,\delta] $, $ G,H \in I(\Omega) $ and $ P \in \Omega $, then $ P $ divides $ G $ and $ H $ on the right. Thus $ P $ also divides $ FG $ and $ G+H $ on the right.
\item
$ F \in I(\{ P \}) $ if, and only if, $ P $ divides $ F $ on the right. 
\item
If $ \mathcal{U} $ satisfies such a property, then $ Z_\mathcal{U}(\{ P \}) $ is formed by all the right divisors of $ P $ of degree at least $ 1 $.
\item[5.]
Let $ F \in I(\Omega_2) $ and $ P \in \Omega_1 $. Since $ \Omega_1 \leq \Omega_2 $, then there is $ Q \in \Omega_2 $ such that $ P $ divides $ Q $ on the right. Since $ F \in I(\Omega_2) $, then $ Q $ divides $ F $ on the right, hence $ P $ divides $ F $ on the right and $ F \in I(\Omega_1) $.
\end{enumerate}
The remaining items are trivial from the definitions. Regarding equality in item 9, note that if $ \Omega = Z_\mathcal{U}(B) $, then $ B \subseteq I(Z_\mathcal{U}(B)) = I(\Omega) $ by item 10, thus $ Z_\mathcal{U}(I(\Omega)) \subseteq Z_\mathcal{U}(B) = \Omega $. Since $ \Omega \subseteq Z_\mathcal{U}(I(\Omega)) $, then equality holds. Similarly for the equality in item 10.
\end{proof}

We may now extend Lam and Leroy's concepts of algebraic sets, minimal skew polynomials, P-closed sets, P-independent sets, P-generators and P-bases \cite{lam, algebraic-conjugacy, lam-leroy}.

\begin{definition} [\textbf{Algebraic sets}] \label{def algebraic sets}
We say that a set $ \Omega \subseteq \mathcal{U} $ is algebraic if $ I(\Omega) \neq \{ 0 \} $.
\end{definition}

\begin{definition} [\textbf{Minimal skew polynomial}]
If $ \Omega \subseteq \mathcal{U} $ is algebraic, we define its minimal skew polynomial, denoted $ F_\Omega $, as the unique monic non-zero skew polynomial of minimum possible degree in $ I(\Omega) $. If $ \Omega \subseteq \mathcal{U} $ is not algebraic, we define $ F_\Omega = 0 $.
\end{definition}

For any set $ \Omega \subseteq \mathcal{U} $, the skew polynomial $ F_\Omega $ exists and generates $ I(\Omega) $. 

In order to bound the degree of the minimal skew polynomial in Proposition \ref{prop deg of min skew pol upper bound by sum}, we will use the following relation, see \cite{ore} or \cite[Deg. Eq. 4.1]{lam-leroy-wedI}.

\begin{lemma} [\textbf{\cite{ore}}] \label{lemma ores cross cut formula}
Given skew polynomials $ F,G \in \mathbb{F}[x;\sigma,\delta] $, let $ D, M \in \mathbb{F}[x;\sigma,\delta] $ be their greatest right common divisor and least left common multiple. Then
$$ \deg(D) + \deg(M) = \deg(F) + \deg(G). $$
\end{lemma}

The first step in extending the concept of P-independence is the following proposition, which extends \cite[Prop. 6]{lam} to high-degree zeros.

\begin{proposition} \label{prop deg of min skew pol upper bound by sum}
Given a non-empty set $ \Omega \subseteq \mathcal{U} $, it holds that
$$ \deg(F_\Omega) \leq \sum_{P \in \Omega} \deg(P). $$
\end{proposition}
\begin{proof}
First, if $ \Omega $ is not finite, then the right-hand side is infinite (recall that $ \deg(P) \geq 1 $ for all $ P \in \mathcal{U} $), hence the inequality holds. Second, if $ \Omega $ is not algebraic, then $ F_\Omega = 0 $ and $ \deg(F_\Omega) = - \infty $, hence the inequality also holds. Finally assume that $\Omega $ is finite and algebraic. By definition, $ F_\Omega $ is the least left common multiple of the skew polynomials in $ \Omega $. Therefore, the result follows from Lemma \ref{lemma ores cross cut formula} applied recursively.
\end{proof}

We now extend the notion of \textit{P-independence} \cite[Sec. 4]{lam} to high-degree points.

\begin{definition} [\textbf{P-independence}] \label{def P-independence}
$ \Omega \subseteq \mathcal{U} $ is P-independent if it is finite, algebraic and
$$ \deg(F_\Omega) = \sum_{P \in \Omega} \deg(P). $$
We also say that a subset $ \Psi = \{ a_1 , a_2, \ldots, a_n \} \subseteq \mathbb{F} $ is P-independent if the corresponding set $ \{ x-a_1, x-a_2, \ldots, x-a_n \} \subseteq \mathcal{U}_0 $ is P-independent.
\end{definition}

P-independent sets are precisely those for which we may obtain an upper bound on the number of zeros of a skew polynomial by its degree (see also Theorems \ref{th deg bound for nonconj naive case} and \ref{th upper bound multiplicities by degree II}).

\begin{theorem} [\textbf{Degree bound}] \label{th upper bound multiplicities by degree general}
An algebraic set $ \Omega \subseteq \mathcal{U} $ is P-independent if, and only if, for every non-zero $ F \in I(\Omega) $, it holds that
\begin{equation}
\sum_{P \in \Omega} \deg(P) \leq \deg(F).
\label{eq degree bound for proof}
\end{equation}
Moreover, when this is the case, then equality in (\ref{eq degree bound for proof}) is attained by $ F $ if, and only if, $ F = a F_\Omega $, where $ a \in \mathbb{F}^* $.
\end{theorem}
\begin{proof}
First assume that $ \Omega $ is P-independent. By definition, $ F_\Omega \neq 0 $ and if $ F \in I(\Omega) $, then $ F_\Omega $ divides $ F $ on the right. By definition of P-independence, we conclude that
$$ \sum_{P \in \Omega} \deg(P) = \deg(F_\Omega) \leq \deg(F). $$

Conversely, assume that $ \Omega $ is not P-independent. Then (\ref{eq degree bound for proof}) does not hold for $ F_\Omega \in I(\Omega) $ by definition (note that $ F_\Omega \neq 0 $ since $ \Omega $ is algebraic).
\end{proof}
%
%

We may now define \textit{P-closed sets}.

\begin{definition} [\textbf{P-closed sets}] \label{def P-closed sets}
For a given set $ \Omega \subseteq \mathcal{U} $, we define its P-closure (in $ \mathcal{U} $) as $ \overline{\Omega} = Z_\mathcal{U}(I(\Omega)) $. We say that $ \Omega $ is P-closed (in $ \mathcal{U} $) if $ \overline{\Omega} = \Omega $.
\end{definition}
%
%
%
%

The next concepts to generalize are those of P-generators and P-bases.

\begin{definition} [\textbf{P-bases}] \label{def P-bases}
Let $ \Omega \subseteq \mathcal{U} $ be a P-closed set. We say that $ \mathcal{G} \leq \Omega $ is a set of P-generators for $ \Omega $ if $ \Omega = \overline{\mathcal{G}} $. Finally, we say that $ \mathcal{B} \leq \Omega $ is a P-basis of $ \Omega $ if $ \mathcal{B} $ is P-independent and a set of P-generators of $ \Omega $.
\end{definition}

The following lemma ensures that all P-bases have the same ``size'' (counting degrees). It follows from the definitions and Proposition \ref{prop properties alg geom}.

\begin{lemma} \label{lemma P-generators iff min skew pols equal}
Given a P-closed set $ \Omega \subseteq \mathcal{U} $ and a set $ \mathcal{G} \leq \Omega $, then $ \mathcal{G} $ is a set of P-generators of $ \Omega $ if, and only if, $ F_\mathcal{G} = F_\Omega $. In particular, if $ \mathcal{B} \leq \Omega $ is a P-basis of $ \Omega $, then 
$$ \deg(F_\Omega) = \sum_{P \in \mathcal{B}} \deg(P) . $$ 
The number $ {\rm Rk}(\Omega) = deg(F_\Omega) $ will be called the rank of $ \Omega $.
\end{lemma}

We next obtain a general Lagrange interpolation theorem. This result will be a prequel for a Hermite interpolation theorem (see Theorems \ref{th hermite interpolation nonconj naive case} and \ref{th hermite interpolation II}). 

\begin{theorem} [\textbf{Lagrange interpolation}] \label{th lagrange interpolation}
Let $ \Omega \subseteq \mathcal{U} $ be a finite algebraic set given by $ \Omega = \{ P_1, P_2, \ldots, P_n \} $, of size $ n $, and define $ N = \deg(F_\Omega) $. The following are equivalent:
\begin{enumerate}
\item
$ \Omega $ is P-independent.
\item
The left $ \mathbb{F} $-linear map
$$ \phi : \mathbb{F}[x;\sigma, \delta]_N \longrightarrow \prod_{i=1}^n \frac{\mathbb{F}[x;\sigma,\delta]}{(P_i)} $$
given by $ \phi (F) = \left( F(P_i) \right)_{i=1}^n $, for $ F \in \mathbb{F}[x;\sigma, \delta]_N $, is a left vector space isomorphism.
\item
The left $ \mathbb{F} $-linear map
$$ \psi : \mathbb{F}[x;\sigma, \delta] \longrightarrow \prod_{i=1}^n \frac{\mathbb{F}[x;\sigma,\delta]}{(P_i)} $$
given by $ \psi (F) = \left( F(P_i) \right)_{i=1}^n $, for $ F \in \mathbb{F}[x;\sigma, \delta] $, is surjective.
\end{enumerate}
Here, $ \mathbb{F}[x;\sigma, \delta]_N = \{ F \in \mathbb{F}[x;\sigma, \delta] \mid \deg(F) < N \} $ and $ \frac{\mathbb{F}[x;\sigma,\delta]}{(P)} $ denotes the left quotient module of $ \mathbb{F}[x;\sigma,\delta] $ by the left ideal $ (P) $, for $ P \in \mathbb{F}[x;\sigma,\delta] $.
\end{theorem}
\begin{proof}
Before proving the corresponding implications, we show that the map $ \phi $ is always injective. Let $ F \in \mathbb{F}[x;\sigma,\delta]_N $ and assume that $ \phi(F) = 0 $. Then $ F \in I(\Omega) $, and therefore, $ F_\Omega $ divides $ F $ on the right. If $ F \neq 0 $, then we would have that $ \deg(F) \geq \deg(F_\Omega) = N $, which is absurd since $ F \in \mathbb{F}[x;\sigma,\delta]_N $. Thus $ F = 0 $ and $ \phi $ is injective. 

We now prove the following equivalences:

$ 1. \Longleftrightarrow 2. $: By counting dimensions, it is trivial that item 2 implies item 1. Assume now that item 1 holds, that is, $ \Omega $ is P-independent. Then the domain and co-domain of $ \phi $ have the same dimension on the left over $ \mathbb{F} $. Since $ \phi $ is always injective, we conclude that $ \phi $ is a left vector space isomorphism.

$ 2. \Longleftrightarrow 3. $: It is trivial that item 2 implies item 3. Assume now that item 3 holds. Since $ \phi $ is always injective, we only need to show that, given $ F \in \mathbb{F}[x;\sigma,\delta] $, there exists $ G \in \mathbb{F}[x;\sigma,\delta]_N $ such that $ F(P_i) = G(P_i) $, for $ i = 1,2, \ldots, n $. To achieve this, let $ G $ be the remainder of $ F $ divided by $ F_\Omega $ on the right (note that $ F_\Omega \neq 0 $ since $ \Omega $ is algebraic).
\end{proof}
%
%

We now deduce the following monotonicity property of P-independence, which will be useful later (see Proposition \ref{prop monotonicity multiplicity I} and Theorem \ref{th P-independence multiplicity seqs}).

\begin{corollary} [\textbf{Monotonicity}] \label{cor monotonicity}
Let $ \Psi, \Omega \subseteq \mathcal{U} $ be finite non-empty sets such that $ \Psi \leq \Omega $. If $ \Omega $ is P-independent, then so is $ \Psi $.
\end{corollary}
\begin{proof}
First observe that, if $ \Omega $ is algebraic, then so is $ \Psi $ by Proposition \ref{prop properties alg geom}. By Definition \ref{def subseteq leq}, if we set $ \Psi = \{ Q_1, Q_2, \ldots, Q_n \} $ of size $ n $, then there exists a subset $ \{ P_1, P_2, \ldots, P_n \} \subseteq \Omega $ of size $ n $, such that $ Q_i $ divides $ P_i $ on the right, for $ i = 1,2, \ldots, n $. In particular, we have that $ (P_i) \subseteq (Q_i) $, for $ i = 1,2, \ldots, n $. Therefore, we have a natural surjective map
$$ \prod_{i=1}^n \frac{\mathbb{F}[x;\sigma, \delta]}{(P_i)} \longrightarrow \prod_{i=1}^n \frac{\mathbb{F}[x;\sigma, \delta]}{(Q_i)} . $$
Hence if $ \Omega $ is P-independent, then so is $ \Psi $ by item 3 in Theorem \ref{th lagrange interpolation}.
\end{proof}
%
%
%
%

\section{Hasse derivatives of skew polynomials} \label{sec hasse derivatives}

Throughout this manuscript, given $ r \in \mathbb{Z}_+ $ and a sequence $ \mathbf{a} = (a_1, a_2, \ldots , a_r) \in \mathbb{F}^r $, we define its associated skew polynomial as
\begin{equation}
P_\mathbf{a} = (x-a_r) \cdots (x-a_2) (x - a_1) \in \mathbb{F}[x;\sigma,\delta] .
\label{eq def multiplicity seq polynomial}
\end{equation}

In this section, we introduce a skew polynomial extension of Hasse derivatives \cite{hasse}. The definition is given so that a point is a zero of a skew polynomial with multiplicity $ r \in \mathbb{Z}_+ $ if, and only if, the corresponding first $ r $ consecutive Hasse derivatives of the skew polynomial at that point are zero. This is not the case with previous notions of derivatives of skew polynomials \cite{bolotnikov-conf, eric}  if $ \mathbb{F} $ has positive characteristic.
%

\begin{definition} [\textbf{Hasse derivative}] \label{def derivative}
Given $ F \in \mathbb{F}[x;\sigma,\delta] $, $ r \in \mathbb{Z}_+ $, and $ \mathbf{a} = (a_1, a_2, $ $ \ldots, a_r) \in \mathbb{F}^r $, we define the corresponding Hasse derivative of order $ r $ as the unique $ D_{\mathbf{a}} (F) \in \mathbb{F} $ such that there exist $ G, H \in \mathbb{F}[x;\sigma,\delta] $, where $ H $ is monic of degree $ r-1 $ and $ F = G P_{\mathbf{a}} + D_{\mathbf{a}} (F) H $, for $ P_{\mathbf{a}} \in \mathbb{F}[x;\sigma,\delta] $ as in (\ref{eq def multiplicity seq polynomial}). 
\end{definition}

The existence and uniqueness of Hasse derivatives, and the following proposition, follow directly from right Euclidean division.
%

\begin{proposition} [\textbf{Taylor expansion}] \label{prop taylor expansion}
Let $ r \in \mathbb{Z}_+ $ and $ \mathbf{a} = (a_1, a_2, \ldots, a_r) \in \mathbb{F}^r $. For every skew polynomial $ F \in \mathbb{F}[x;\sigma,\delta] $, there exists $ G \in \mathbb{F}[x;\sigma,\delta] $ such that
$$ F = G P_{\mathbf{a}_r} + D_{\mathbf{a}_{r}} (F) P_{\mathbf{a}_{r-1}} + \cdots + D_{\mathbf{a}_{2}} (F) P_{\mathbf{a}_1} + D_{\mathbf{a}_1} (F) , $$ 
where $ \mathbf{a}_i = (a_1, a_2, \ldots, a_i) \in \mathbb{F}^i $ and $ P_{\mathbf{a}_i} \in \mathbb{F}[x;\sigma,\delta] $ is as in (\ref{eq def multiplicity seq polynomial}), for $ i = 1,2, \ldots, r $. In particular, we have that $ D_{\mathbf{a}_1}(F) = F(a_1) \in \mathbb{F} $ (see Definition \ref{def evaluation high degree}).
%
\end{proposition}
%

We conclude with the corresponding \textit{confluent Vandermonde matrices}. 

\begin{definition} [\textbf{Confluent Vandermonde matrices}] \label{def confluent vandermonde}
Let $ n, r_i \in \mathbb{Z}_+ $ and $ \mathbf{a}_i = (a_{i,1}, a_{i,2}, $ $ \ldots , $ $ a_{i,r_i}) $ $ \in \mathbb{F}^{r_i} $, and define $ \mathbf{a}_{i,j} = (a_{i,1}, a_{i,2} , \ldots, a_{i,j}) \in \mathbb{F}^j $, for $ j = 1,2, \ldots r_i $ and $ i = 1,2, \ldots, n $. Given $ N \in \mathbb{Z}_+ $, define the confluent Vandermonde matrix of order $ N $ on $ \mathbf{a}_i $ as the $ N \times r_i $ matrix with entries in $ \mathbb{F} $ given by
$$ V_N (\mathbf{a}_i) = \left( \begin{array}{cccc}
D_{\mathbf{a}_{i,1}} (1) & D_{\mathbf{a}_{i,2}} (1) & \ldots & D_{\mathbf{a}_{i,r_i}} (1) \\
D_{\mathbf{a}_{i,1}} (x) & D_{\mathbf{a}_{i,2}} (x) & \ldots & D_{\mathbf{a}_{i,r_i}} (x) \\
\vdots & \vdots & \ddots & \vdots \\
D_{\mathbf{a}_{i,1}} (x^{N-1}) & D_{\mathbf{a}_{i,2}} (x^{N-1}) & \ldots & D_{\mathbf{a}_{i,r_i}} (x^{N-1}) \\
\end{array} \right) , $$
for $ i = 1,2, \ldots, n $. Finally, we define the confluent Vandermonde matrix of order $ N $ on $ \mathbf{a}_1, \mathbf{a}_2, \ldots, $ $ \mathbf{a}_n $ as the $ N \times (\sum_{i=1}^n r_i) $ matrix with entries in $ \mathbb{F} $ given by 
$$ V_N(\mathbf{a}_1, \mathbf{a}_2, \ldots, \mathbf{a}_n) = \left( V_N (\mathbf{a}_1), V_N (\mathbf{a}_2), \ldots, V_N (\mathbf{a}_n) \right) . $$
\end{definition}

Classical confluent Vandermonde matrices \cite{kalman} are recovered from Definition \ref{def confluent vandermonde} when $ \sigma = {\rm Id} $, $ \delta = 0 $ and $ a_{i,1} = a_{i,2} = \ldots = a_{i,r_i} $, for $ i = 1,2, \ldots, n $. Skew polynomial Vandermonde matrices \cite{lam, lam-leroy} are recovered from Definition \ref{def confluent vandermonde} when $ r_1 = r_2 = \ldots = r_n = 1 $. Finally, quaternionic confluent Vandermonde matrices \cite{bolotnikov-conf} are recovered from Definition \ref{def confluent vandermonde} when $ \mathbb{F} $ denotes the quaternions, $ \sigma = {\rm Id} $, $ \delta = 0 $ and $ \mathbf{a}_i $ is a spherical chain (what we will call a multiplicity sequence in Definition \ref{def multiplicity seq}), for $ i = 1,2, \ldots, n $. 

\section{Multiplicities from powers of a linear polynomial} \label{sec multi from powers}

In this section, we provide a first definition of multiplicities of zeros of skew polynomials (similar to \cite{eric, gentili-zero, johnson-thesis}). We consider a different definition in Sections \ref{sec multi from linear terms}, \ref{sec multiplicity seqs} and \ref{sec particular cases}. 
%
%

\begin{definition} [\textbf{Multiplicity I}] \label{def multiplicity I}
Given $ F \in \mathbb{F}[x;\sigma,\delta] $, $ a \in \mathbb{F} $ and $ r \in \mathbb{Z}_+ $, we say that $ a $ is a zero of $ F $ of multiplicity $ r $ if $ (x-a)^r \in \mathbb{F}[x;\sigma,\delta] $ divides $ F $ on the right.
\end{definition}

We may now characterize multiplicities (as in Definition \ref{def multiplicity I}) in terms of Hasse derivatives (Definition \ref{def derivative}). This result follows from Proposition \ref{prop taylor expansion}. 

\begin{proposition} [\textbf{Derivative criterion I}] \label{prop charact multiplicities derivatives I}
Given $ F \in \mathbb{F}[x;\sigma,\delta] $ and $ a \in \mathbb{F} $, denote by $ F^{(i)} (a) = D_{\mathbf{a}_{i+1}}(F) \in \mathbb{F} $ the Hasse derivative (Definition \ref{def derivative}) of $ F $ at $ \mathbf{a}_{i+1} = (a,a,\ldots, a) \in \mathbb{F}^{i+1} $, for $ i \in \mathbb{N} $. Then $ a $ is a zero of $ F $ of multiplicity $ r \in \mathbb{Z}_+ $ (according to Definition \ref{def multiplicity I}) if, and only if, $ F^{(0)}(a) = F^{(1)}(a) = \ldots = F^{(r-1)}(a) = 0 $.
\end{proposition}

We now use the general framework from Section \ref{sec high-degree zeros} for the type of multiplicities from Definition \ref{def multiplicity I}. To this end, we consider the universal set $ \mathcal{U}_1 $ from (\ref{eq universal set I}). 

By Theorems \ref{th upper bound multiplicities by degree general} and \ref{th lagrange interpolation}, we see that bounding the number of zeros with multiplicity and Hermite interpolation are equivalent to considering P-independent subsets of $ \mathcal{U}_1 $. Therefore, the real challenge is to identify such P-independent subsets. A first result is the particular case of monotonicity (Corollary \ref{cor monotonicity}) applied to $ \mathcal{U}_1 $. 

\begin{proposition} [\textbf{Monotonicity}] \label{prop monotonicity multiplicity I}
Let $ \Omega = \{ (x-a_1)^{r_1}, (x-a_2)^{r_2}, \ldots, (x-a_n)^{r_n} \} \subseteq \mathbb{F}[x;\sigma,\delta] $ be of size $ n \in \mathbb{Z}_+ $, where $ a_1,a_2,\ldots, a_n \in \mathbb{F} $ and $ r_1, r_2,\ldots, r_n \in \mathbb{Z}_+ $. 
\begin{enumerate}
\item
If $ \Omega $ is P-independent, then so is $ \Psi = \{ (x-a_1)^{s_1}, (x-a_2)^{s_2}, \ldots, (x-a_n)^{s_n} \} $ $ \subseteq \mathbb{F}[x;\sigma,\delta] $, for all $ s_1, s_2, \ldots, s_n \in \mathbb{Z}_+ $ such that $ s_i \leq r_i $, for $ i = 1,2, \ldots, n $. 
\item
If $ \Omega $ is P-independent, then the underlying set $ \Psi = \{ a_1, a_2, \ldots, a_n \} $ $ \subseteq \mathbb{F} $ has size $ n $ and is P-independent.
\end{enumerate}
\end{proposition}

However, the reversed implications do not hold, as the following example shows.

\begin{example} \label{ex power of linear term is no multiplicity}
Let $ q $ be a power of an odd prime number and let $ m \in \mathbb{N} $ be an even number with $ m \geq 2 $. Let $ \mathbb{F} = \mathbb{F}_{q^m} $ be the finite field with $ q^m $ elements, let $ \sigma : \mathbb{F}_{q^m} \longrightarrow \mathbb{F}_{q^m} $ be given by $ \sigma(a) = a^q $, for $ a \in \mathbb{F}_{q^m} $, and let $ \delta = 0 $. Let $ \gamma \in \mathbb{F}_{q^m}^* $ be a primitive element (a generator of the cyclic multiplicative group $ \mathbb{F}_{q^m}^* $ \cite[Th. 2.8]{lidl}) and consider
$$ a = \gamma^{ \frac{1}{2} \cdot \frac{q^m - 1}{q-1}} \quad \textrm{and} \quad b = - \gamma^{ \frac{1}{2} \cdot \frac{q^m - 1}{q-1}} . $$
Since $ q $ is odd and $ m $ is even, $ 2 $ divides $ (q^m-1)/(q-1) = 1 + q + q^2 + \cdots + q^{m-1} $, thus $ a $ and $ b $ are well defined, and since $ q $ is odd, it holds that $ a \neq b $. We have
$$ (x-a)^2 = (x-b)^2 \in \mathbb{F}_{q^m}[x;\sigma,\delta]. $$
Thus the doubleton set $ \Psi = \{ a,b \} \subseteq \mathbb{F}_{q^m} $ is P-independent, but the doubleton set $ \Omega = \{ (x-a)^2, (x-b)^3 \} \subseteq \mathbb{F}_{q^m}[x;\sigma,\delta] $ is not P-independent.

This example also shows that the minimal skew polynomial of $ \Psi = \{ a,b \} $ is $ F_\Psi = (x-a)^2 = (x-b)^2 $, which has zeros of multiplicity $ 2 $, that is, its zeros are not all simple. 

This is according to Definition \ref{def multiplicity I}. Multiplicities as in Definition \ref{def multiplicity II} in the following section do not suffer from such pathologies (see Theorem \ref{th P-independence multiplicity seqs}).
\end{example}


However, a union theorem \cite[Th. 22]{lam} still holds in general for multiplicities as in Definition \ref{def multiplicity I}. To that end, we need some preliminary definitions and lemmas. The first of these is the equivalence relation given by \textit{conjugacy} \cite[Eq. (2.5)]{lam-leroy}.
%

\begin{definition} [\textbf{Conjugacy \cite{lam-leroy}}] \label{def conjugate}
Given $ a,b \in \mathbb{F} $, we say that $ b $ is a $ (\sigma,\delta) $-conjugate (or simply conjugate) of $ a $ if there exists $ \beta \in \mathbb{F}^* $ such that $ b = a^\beta $, where 
$$ a^\beta = \sigma(\beta) a \beta^{-1} + \delta(\beta) \beta^{-1}. $$
The conjugacy class of $ a $ is defined as $ C^{\sigma,\delta}(a) = \left\lbrace a^\beta \mid \beta \in \mathbb{F}^* \right\rbrace $.
\end{definition}

Conjugacy allows us to introduce the powerful product rule \cite[Th. 2.7]{lam-leroy}.

\begin{lemma} [\textbf{Product rule \cite{lam-leroy}}] \label{lemma product rule}
Let $ F,G \in \mathbb{F}[x;\sigma,\delta] $ and let $ a \in \mathbb{F} $. If $ G(a) = 0 $, then $ (FG)(a) = 0 $. On the other hand, if $ \beta = G(a) \neq 0 $, then
$$ (FG)(a) = F (a^\beta) G(a). $$
\end{lemma}

Finally, we need the following auxiliary lemma.

\begin{lemma} \label{lemma min skew pol expression for union}
Assume that $ \Omega = \{ (x-a_1)^{r_1}, (x-a_2)^{r_2}, \ldots, (x-a_n)^{r_n} \} \subseteq \mathbb{F}[x;\sigma,\delta] $ is P-independent of size $ n \in \mathbb{Z}_+ $, where $ a_1,a_2,\ldots, a_n \in \mathbb{F} $ and $ r_1, r_2,\ldots, r_n \in \mathbb{Z}_+ $. Then there exist elements $ \beta_{i,j} \in \mathbb{F}^* $, for $ j = 1,2, \ldots, r_i $ and $ i = 1,2, \ldots, n $, such that, if $ \Omega_{i,j} = \{ (x-a_1)^{r_1}, (x-a_2)^{r_2}, \ldots, (x-a_{i-1})^{r_{i-1}}, (x-a_i)^{j} \} $ (thus $ \Omega_{n,r_n} = \Omega $), then
\begin{equation}
F_{\Omega_{i,j}} = \left( x-a_i^{\beta_{i,j}} \right) \cdots \left(x-a_i^{\beta_{i,1}} \right) \cdots \left(x-a_1^{\beta_{1,r_1}} \right) \cdots \left(x-a_1^{\beta_{1,1}} \right) ,
\label{eq betas for min skew pol of multiplicities condition}
\end{equation}
for $ j = 1,2, \ldots, r_i $ and $ i = 1,2, \ldots, n $. 
\end{lemma}
\begin{proof}
We prove the lemma by induction on $ r = r_1 + r_2 + \cdots + r_n $. First, the case $ n = r_1 = r= 1 $ is trivial. Assume now that $ r \geq 2 $. We will assume for convenience that $ r_n \geq 2 $, but the case $ r_n = 1 $ (still with $ r \geq 2 $) is proven analogously.

Define $ \Psi = \{ (x-a_1)^{r_1}, (x-a_2)^{r_2}, \ldots, (x-a_n)^{r_n - 1} \} $, which is P-independent by Corollary \ref{cor monotonicity} since $ \Psi \leq \Omega $. Furthermore, $ \Psi $ has size $ n $ since $ \Omega $ is P-independent. By induction hypothesis, there exist $ \beta_{i,j} \in \mathbb{F}[x;\sigma,\delta] $ satisfying Equation (\ref{eq betas for min skew pol of multiplicities condition}), for $ j = 1,2, \ldots, r_i $ if $ 1 \leq i \leq n-1 $, and for $ j = 1,2, \ldots, r_n-1 $ if $ i = n $. In particular, 
$$ F_\Psi = \left( x-a_n^{\beta_{n,r_n-1}} \right) \cdots \left(x-a_n^{\beta_{n,1}} \right) \cdots \left(x-a_1^{\beta_{1,r_1}} \right) \cdots \left(x-a_1^{\beta_{1,1}} \right). $$
Since $ \Psi \leq \Omega $, then $ F_\Psi $ divides $ F_\Omega $ on the right. Furthermore, by counting degrees, since $ \Omega $ is P-independent, there exists $ b \in \mathbb{F} $ with $ F_\Omega = (x-b) F_\Psi $. Now, by definition of $ F_\Psi $ and since $ \Omega $ is P-independent, we have that $ (x-a_n)^{r_n-1} $ divides $ F_\Psi $ on the right, but $ (x-a_n)^{r_n} $ does not. That is, there exists $ G \in \mathbb{F}[x;\sigma,\delta] $ such that 
$$ F_\Psi = G (x-a_n)^{r_n-1} \quad \textrm{but} \quad G(a_n) \neq 0. $$
Moreover, since $ (x-a_n)^{r_n} $ divides $ F_\Omega $ on the right, there exists a skew polynomial $ H \in \mathbb{F}[x;\sigma,\delta] $ such that $ F_\Omega = (x-b) G (x-a_n)^{r_n-1} = H (x-a_n)^{r_n} $. Thus, we deduce that $ (x-b) G = H (x-a_n) $. In particular, $ a_n $ is a zero of $ (x-b) G $ but not of $ G $. By the product rule (Lemma \ref{lemma product rule}), we deduce that $ G(a_n) \neq 0 $ and $ a_n^{G(a_n)} = b $. Therefore, we conclude by choosing $ \beta_{n,r_n} = G(a_n) $.
\end{proof}

The following theorem extends \cite[Th. 22]{lam} to multiplicities as in Definition \ref{def multiplicity I}.

\begin{theorem} [\textbf{Union}] \label{th union}
Let $ \Omega = \{ (x-a_1)^{r_1}, \ldots, (x-a_n)^{r_n} \} $ and $ \Psi = \{ (x-b_1)^{s_1}, \ldots, (x-b_m)^{s_m} \} $ be P-independent sets of sizes $ n,m \in \mathbb{Z}_+ $, respectively, and where $ a_1,\ldots, a_n, b_1, $ $ \ldots, $ $ b_m $ $ \in \mathbb{F} $ and $ r_1, \ldots, r_n, s_1, \ldots, s_m \in \mathbb{Z}_+ $. If no element in $ \{ a_1, a_2, \ldots, a_n \} $ is conjugate to an element in $ \{ b_1, b_2, \ldots, b_m\} $, then $ \Omega \cup \Psi $ has size $ n+m $ and is P-independent.
\end{theorem}
\begin{proof}
We proceed by induction on $ t = \sum_{i=1}^n r_i + \sum_{j=1}^m s_j $. Since $ n \geq 1 $ and $ m \geq 1 $, then $ t \geq 2 $. Observe that the case $ r_1 = \ldots = r_n = s_1 = \ldots = s_m = 1 $ is precisely Lam's union theorem \cite[Th. 22]{lam}. Hence the basis step $ t = 2 $ follows (since, in that case, $ n = m = r_1 = s_1 = 1 $), and for the inductive step, we may assume without loss of generality that $ s_m \geq 2 $. In other words, we may define 
$$ \Psi^\prime = \{ (x-b_1)^{s_1}, (x-b_2)^{s_2}, \ldots, (x-b_m)^{s_m -1} \}, $$
which is P-independent by Corollary \ref{cor monotonicity} since $ \Psi^\prime \leq \Psi $. By induction hypothesis, $ \Omega \cup \Psi^\prime $ is P-independent. We proceed now by contradiction and assume that $ \Omega \cup \Psi $ is not P-independent. Since $ \Omega \cup \Psi^\prime \leq \Omega \cup \Psi $, then $ F_{\Omega \cup \Psi^\prime} $ divides $ F_{\Omega \cup \Psi} $ on the right. By counting degrees, we conclude that $ F_{\Omega \cup \Psi^\prime} = F_{\Omega \cup \Psi} $. By Lemma \ref{lemma min skew pol expression for union}, there exist elements $ \beta_{i,j} \in \mathbb{F}^* $, for $ j = 1,2, \ldots, r_i $ and $ i = 1,2, \ldots, n $, such that
$$ F_{\Omega \cup \Psi} = F_{\Omega \cup \Psi^\prime} = \left( x-a_n^{\beta_{n,r_n}} \right) \cdots \left(x-a_n^{\beta_{n,1}} \right) \cdots \left(x-a_1^{\beta_{1,r_1}} \right) \cdots \left(x-a_1^{\beta_{1,1}} \right) F_{\Psi^\prime}. $$
Since $ \Psi $ is P-independent, then $ (x-b_m)^{s_m-1} $ divides $ F_{\Psi^\prime} $ on the right, but $ (x-b_m)^{s_m} $ does not. Hence there exists a skew polynomial $ G \in \mathbb{F}[x;\sigma,\delta] $ such that $ F_{\Psi^\prime} = G (x-b_m)^{s_m-1} $ and $ G(b_m) \neq 0 $. Thus we deduce that $ b_m $ is a zero of the skew polynomial
$$ \left( x-a_n^{\beta_{n,r_n}} \right) \cdots \left(x-a_n^{\beta_{n,1}} \right) \cdots \left(x-a_1^{\beta_{1,r_1}} \right) \cdots \left(x-a_1^{\beta_{1,1}} \right) G, $$
while $ G(b_m) \neq 0 $. By the product rule (Lemma \ref{lemma product rule}), we deduce that there exists $ \gamma \in \mathbb{F}^* $ and indices $ j = 1,2, \ldots, r_i $ and $ i = 1,2, \ldots, n $, such that $ b_m^\gamma = a_i^{\beta_{i,j}} $. This contradicts the hypothesis that no element in $ \{ a_1, a_2, \ldots, a_n \} $ is conjugate to an element in $ \{ b_1, b_2, \ldots, b_m\} $, and we are done.
\end{proof}
%

We may now extend the degree bound and Hermite interpolation. First we combine Theorem \ref{th union} with the fact that singleton sets are always P-independent.
%

\begin{corollary} \label{cor non-conj multiple points are P-indep}
Let $ \Omega = \{ (x-a_1)^{r_1}, (x-a_2)^{r_2}, \ldots, (x-a_n)^{r_n} \} \subseteq \mathbb{F}[x;\sigma,\delta] $, where $ a_1,a_2,\ldots, a_n \in \mathbb{F} $ and $ r_1, r_2,\ldots, r_n \in \mathbb{Z}_+ $. If $ a_1, a_2, \ldots, a_n $ are pair-wise non-conjugate, then $ \Omega $ has size $ n $ and is P-independent.
\end{corollary}

The next result combines Corollary \ref{cor non-conj multiple points are P-indep} with Theorem \ref{th upper bound multiplicities by degree general}.

\begin{theorem} [\textbf{Degree bound I}] \label{th deg bound for nonconj naive case}
Let $ a_1, a_2, \ldots, a_n \in \mathbb{F} $ be pair-wise non-conjugate. If $ F \in \mathbb{F}[x;\sigma,\delta] $ is not zero and has $ a_i $ as a zero of multiplicity $ r_i \in \mathbb{Z}_+ $ (according to Definition \ref{def multiplicity I}), for $ i = 1,2, \ldots, n $, then it holds that
$$ \sum_{i=1}^n r_i \leq \deg(F). $$
Furthermore, equality holds if, and only if, $ F = a F_\Omega $, where $ a \in \mathbb{F}^* $ and $ F_\Omega $ is the least left common multiple of $ (x-a_1)^{r_1} , (x-a_2)^{r_2}, \ldots, (x-a_n)^{r_n} $.
\end{theorem}
%

The next results follows from Corollary \ref{cor non-conj multiple points are P-indep} with Theorem \ref{th lagrange interpolation}.

\begin{theorem} [\textbf{Hermite interpolation I}] \label{th hermite interpolation nonconj naive case}
Let $ a_1, a_2, \ldots, a_n \in \mathbb{F} $ be pair-wise non-conjugate. Let the notation of Hasse derivatives be as in Proposition \ref{prop charact multiplicities derivatives I}. For all $ b_{i,j} \in \mathbb{F} $, for $ j = 1,2,\ldots, r_i $ and $ i = 1,2, \ldots, n $, there exists a unique $ F \in \mathbb{F}[x;\sigma,\delta] $ with $ \deg(F) < \sum_{i=1}^n r_i $ such that $ F^{(j-1)}(a_i) = b_{i,j} $, for $ j = 1,2,\ldots, r_i $ and $ i = 1,2, \ldots, n $. In particular, if $ N = \sum_{i=1}^n r_i $ and $ \mathbf{a}_i = (a_i, a_i, \ldots, a_i) \in \mathbb{F}^{r_i} $, for $ i = 1,2, \ldots, n $, then the confluent Vandermonde matrix $ V_N(\mathbf{a}_1, \mathbf{a}_2, \ldots, \mathbf{a}_n) $ from Definition \ref{def confluent vandermonde} is invertible.
\end{theorem}

\section{Multiplicities from linear products with a single zero} \label{sec multi from linear terms}

In this section, we turn to a different notion of multiplicities, which allows a set of multiple zeros to be P-independent if, and only if, the underlying set of zeros is P-independent (see Theorem \ref{th P-independence multiplicity seqs}). The following definition is inspired by Bolotnikov's works on \textit{spherical chains} of quaternions (see \cite[Lemma 4.1]{bolotnikov-zeros} or \cite[Th. 3.1]{bolotnikov-conf}). 
%

\begin{definition} [\textbf{Multiplicity sequences \cite{bolotnikov-zeros}}] \label{def multiplicity seq}
For $ r \in \mathbb{Z}_+ $, we say that $ \mathbf{a} = (a_1, $ $ a_2, \ldots, $ $ a_r) \in \mathbb{F}^r $ is a $ (\sigma,\delta) $-multiplicity sequence if $ a_1 $ is the only zero of
$$ P_{\mathbf{a}} = (x-a_r) \cdots (x-a_2) (x - a_1) \in \mathbb{F}[x;\sigma,\delta] . $$
\end{definition}

In contrast with the commutative case, a $ (\sigma,\delta) $-multiplicity sequence does not need to satisfy $ a_1 = a_2 = \ldots = a_r $. This was first noted by Leroy in \cite[Ex. 1.15.3]{leroy-noncommutative}. In Section \ref{sec multiplicity seqs}, we will provide sufficient and necessary conditions for a sequence to be a $ (\sigma,\delta) $-multiplicity sequence. In particular, the existence and (lack of) uniqueness of $ (\sigma,\delta) $-multiplicity sequences will be discussed in Corollary \ref{cor existence multiplicity seqs} and Section \ref{sec particular cases}.
%

One of the key properties of $ (\sigma,\delta) $-multiplicity sequences is their monotonicity.

\begin{proposition} [\textbf{Monotonicity}] \label{prop monotonicity multiplicity seqs}
If $ (a_1, a_2, \ldots, a_r) \in \mathbb{F}^r $ is a $ (\sigma,\delta) $-multiplicity sequence, then so are the sequences $ (a_1, a_2, \ldots, a_i) \in \mathbb{F}^i $, for $ i = 1,2, \ldots, r-1 $.
\end{proposition}

From now on, we only consider the following definition of multiplicity, which extends Bolotnikov's definition for quaternionic polynomials \cite{bolotnikov-zeros}. 

\begin{definition} [\textbf{Multiplicity II \cite{bolotnikov-zeros}}] \label{def multiplicity II}
Given $ F \in \mathbb{F}[x;\sigma,\delta] $, $ r \in \mathbb{Z}_+ $, and a $ (\sigma,\delta) $-multiplicity sequence $ \mathbf{a} = (a_1, a_2, \ldots, a_r) \in \mathbb{F}^r $, we say that $ a_1 $ is a zero of $ F $ of multiplicity $ r $ via $ \mathbf{a} $ if the skew polynomial $ P_{\mathbf{a}} $ in (\ref{eq def multiplicity seq polynomial}) divides $ F $ on the right (i.e, $ F(P_\mathbf{a}) = 0 $ according to Definition \ref{def evaluation high degree}). We say that $ a \in \mathbb{F} $ is a zero of multiplicity $ r $ of $ F $ if it is a zero of multiplicity $ r $ of $ F $ via some $ (\sigma,\delta) $-multiplicity sequence starting from $ a $.
\end{definition}
%

Similarly to Proposition \ref{prop charact multiplicities derivatives I}, we have the following characterization of multiplicities based on Hasse derivatives (Definition \ref{def derivative}), which follows from Proposition \ref{prop taylor expansion}.

\begin{proposition} [\textbf{Derivative criterion II}] \label{prop charact multiplicities derivatives II}
Given $ F \in \mathbb{F}[x;\sigma,\delta] $, $ r \in \mathbb{Z}_+ $ and a $ (\sigma,\delta) $-multiplicity sequence $ \mathbf{a} = (a_1, a_2, \ldots, a_r) \in \mathbb{F}^r $, it holds that $ a_1 $ is a zero of $ F $ of multiplicity $ r $ via $ \mathbf{a} $ if, and only if, $ D_{\mathbf{a}_1} (F) = D_{\mathbf{a}_2} (F) = \ldots = D_{\mathbf{a}_r} (F) = 0 $, where $ \mathbf{a}_i = (a_1, a_2, \ldots, a_i) \in \mathbb{F}^i $, for $ i = 1,2, \ldots, r $.
\end{proposition}

We may also deduce the following monotonicity property.

\begin{proposition} [\textbf{Monotonicity}]
Let $ F \in \mathbb{F}[x;\sigma,\delta] $, and let $ \mathbf{a} = (a_1, a_2, \ldots, a_r) \in \mathbb{F}^r $ be a $ (\sigma,\delta) $-multiplicity sequence. If $ a_1 $ is a zero of $ F $ of multiplicity $ r $ via $ \mathbf{a} $, then it is also a zero of $ F $ of multiplicity $ i $ via $ \mathbf{a}_i = (a_1, a_2, \ldots, a_i) $, for $ i = 1,2, \ldots, r $.
\end{proposition}

We now use the general framework from Section \ref{sec high-degree zeros} for the type of multiplicities from Definition \ref{def multiplicity II}. To this end, we use in this section the universal set $ \mathcal{U}_2 $ from (\ref{eq universal set II}).
%

We will also use the concept of \textit{centralizers}, introduced in \cite[Eq. (3.1)]{lam-leroy}. 

\begin{definition} [\textbf{Centralizer \cite{lam-leroy}}] \label{def centralizer}
We define the $ (\sigma,\delta) $-centralizer of $ a \in \mathbb{F} $ (or simply centralizer) as the division subring of $ \mathbb{F} $ given by $ K^{\sigma,\delta}_a = \{ \beta \in \mathbb{F} \mid \sigma(\beta) a + \delta(\beta) = a \beta \} $.
\end{definition}

The following characterization is the main result of this section.

\begin{theorem} \label{th P-independence multiplicity seqs}
Let $ \mathbf{a}_i = (a_{i,1}, a_{i,2}, \ldots, a_{i,r_i}) \in \mathbb{F}^{r_i} $ be a $ (\sigma,\delta) $-multiplicity sequence, for $ i = 1,2, \ldots, n $. Assume that $ a_{1,1}, a_{2,1}, \ldots, a_{n,1} $ are distinct. Let also $ a_1, a_2, \ldots , a_\ell \in \mathbb{F} $ be pair-wise non-conjugate elements, $ n_i \in \mathbb{Z}_+ $ and $ \beta_{i,1}, \beta_{i,2}, \ldots, \beta_{i,n_i} \in \mathbb{F}^* $ such that
$$ \{ a_{1,1}, a_{2,1}, \ldots, a_{n,1} \} = \bigcup_{i=1}^\ell  \left\lbrace a_i^{\beta_{i,1}}, a_i^{\beta_{i,2}}, \ldots , a_i^{\beta_{i,n_i}} \right\rbrace , $$
for $ i = 1,2, \ldots, \ell $, and $ n = n_1 + n_2 + \cdots + n_\ell $. The following are equivalent:
\begin{enumerate}
\item
$ \Omega = \{ P_{\mathbf{a}_1}, P_{\mathbf{a}_2}, \ldots, P_{\mathbf{a}_n} \} \subseteq \mathcal{U}_2 $ is P-independent.
\item
$ \Psi = \{ a_{1,1}, a_{2,1}, \ldots, a_{n,1} \} \subseteq \mathbb{F} $ is P-independent.
\item
$ \beta_{i,1}, \beta_{i,2}, \ldots , \beta_{i,n_i} \in \mathbb{F}^* $ are right $ K^{\sigma,\delta}_{a_i} $-linearly independent, for $ i = 1,2, \ldots, \ell $.
\end{enumerate}
\end{theorem}
\begin{proof}
The equivalence between items 2 and 3 is due to Lam and Leroy (see \cite[Cor. 21]{lam} and \cite[Th. 4.5]{lam-leroy}). We now prove the equivalence between items 1 and 2. First, if $ \Omega $ is P-independent, then so is $ \Psi $ by Corollary \ref{cor monotonicity}. Hence, we only need to prove the reversed implication. We proceed by induction on $ \sum_{i=1}^n r_i $, being the case $ n = r_1 = 1 $ trivial. Next, $ \Omega $ and $ \Psi $ are always P-independent if $ n = 1 $. Furthermore, in the case $ r_1 = r_2 = \ldots = r_n = 1 $, it holds by definition that $ \Omega $ is P-independent if, and only if, so is $ \Psi $. Therefore, we may assume without loss of generality that $ n \geq 2 $ and $ r_n \geq 2 $. 

Since $ r_n \geq 2 $, we may define $ \mathbf{a}_n^\prime = (a_{n,1}, a_{n,2}, \ldots, a_{n,r_n-1}) \in \mathbb{F}^{r_n-1} $, which is a $ (\sigma,\delta) $-multiplicity sequence by Proposition \ref{prop monotonicity multiplicity seqs}. Define $ \Omega^\prime = \{ P_{\mathbf{a}_1}, P_{\mathbf{a}_2}, \ldots, P_{\mathbf{a}_{n-1}}, P_{\mathbf{a}_n^\prime} \} $. Since $ n \geq 2 $, we may also define $ \Phi = \{ P_{\mathbf{a}_1}, P_{\mathbf{a}_2}, \ldots, P_{\mathbf{a}_{n-1}} \} $. By induction hypothesis, both $ \Omega^\prime $ and $ \Phi $ are P-independent, which means by Definition \ref{def P-independence} that
$$ \deg(F_{\Omega^\prime}) = \sum_{i=1}^n r_i - 1 \quad \textrm{and} \quad \deg(F_{\Phi}) = \sum_{i=1}^{n-1} r_i. $$

Assume now that $ \Omega $ is not P-independent. Since $ \Omega^\prime \leq \Omega $, then $ F_{\Omega^\prime} $ divides $ F_\Omega $ on the right. By Proposition \ref{prop deg of min skew pol upper bound by sum}, we deduce that $ F_{\Omega^\prime} = F_\Omega $. In other words, 
$$ \deg(F_\Omega) = \sum_{i=1}^n r_i - 1 = \deg(F_\Phi) + \deg(P_{\mathbf{a}_n}) - 1. $$
Observe that, by definition, $ F_\Omega $ is the least left common multiple of $ F_\Phi $ and $ P_{\mathbf{a}_n} $. By Lemma \ref{lemma ores cross cut formula}, we deduce that the greatest right common divisor of $ F_\Phi $ and $ P_{\mathbf{a}_n} $ has degree exactly $ 1 $, that is, it is of the form $ D = x - b \in \mathbb{F}[x;\sigma,\delta] $, for some $ b \in \mathbb{F} $. Now, since $ D $ divides $ P_{\mathbf{a}_n} $ on the right and $ \mathbf{a}_n $ is a $ (\sigma,\delta) $-multiplicity sequence, we conclude that $ b = a_{n,1} $. However, this also means that $ D $ divides $ P_{\mathbf{a}_n^\prime} $ on the right. In other words, $ D $ is a right common divisor of both $ P_{\mathbf{a}_n^\prime} $ and $ F_\Phi $. However, the least left common multiple of $ P_{\mathbf{a}_n^\prime} $ and $ F_\Phi $ is $ F_{\Omega^\prime} $ by definition. Thus, using again Lemma \ref{lemma ores cross cut formula}, we deduce that 
$$ \deg(F_{\Omega^\prime}) \leq \deg(P_{\mathbf{a}_n^\prime}) + \deg(F_\Phi) - 1 = \sum_{i=1}^n r_i - 2, $$
which contradicts the fact that $ \Omega^\prime $ is P-independent, and we are done.
\end{proof}

We now derive the following particular case of Theorem \ref{th upper bound multiplicities by degree general} from Theorem \ref{th P-independence multiplicity seqs}.

\begin{theorem} [\textbf{Degree bound II}] \label{th upper bound multiplicities by degree II}
Let $ \Psi = \{a_1, a_2, \ldots, a_n \} \subseteq \mathbb{F} $ be a set of size $ n \in \mathbb{Z}_+ $ and let $ r_1, r_2, \ldots, r_n \in \mathbb{Z}_+ $. Then $ \Psi $ is P-independent if, and only if, for every non-zero $ F \in \mathbb{F}[x;\sigma,\delta] $ that has $ a_i $ as a zero of multiplicity $ r_i $, for $ i = 1,2, \ldots, n $, it holds that
$$ \sum_{i=1}^n r_i \leq \deg(F). $$
Furthermore, when this is the case, then equality holds for $ F $ if, and only if, $ F = a F_\Omega $, where $ a \in \mathbb{F}^* $, $ F_\Omega $ is the least left common multiple of $ P_{\mathbf{a}_1} , P_{\mathbf{a}_2} ,\ldots, P_{\mathbf{a}_n} $, and $ \mathbf{a}_i \in \mathbb{F}^{r_i} $ is a $ (\sigma,\delta) $-multiplicity sequence starting at $ a_i $, for $ i = 1,2, \ldots, n $.
\end{theorem}

We may analogously rephrase Theorem \ref{th lagrange interpolation} as a Hermite interpolation theorem.

\begin{theorem} [\textbf{Hermite interpolation II}] \label{th hermite interpolation II}
Given a set $ \Psi = \{a_1, a_2, \ldots, a_n \} \subseteq \mathbb{F} $ of size $ n \in \mathbb{Z}_+ $ and given $ r_1, r_2, \ldots, r_n \in \mathbb{Z}_+ $, the following are equivalent:
\begin{enumerate}
\item
$ \Psi $ is P-independent.
\item
For $ b_{i,j} \in \mathbb{F} $ and $ (\sigma,\delta) $-multiplicity sequences $ \mathbf{a}_{i,j} = (a_{i,1}, a_{i,2}, $ $ \ldots, a_{i,j}) $ with $ a_{i,1} = a_i $, for $ j = 1,2, \ldots, r_i $ and $ i = 1,2, \ldots, n $, there exists a unique $ F \in \mathbb{F}[x;\sigma,\delta] $ with $ \deg(F) < \sum_{i=1}^n r_i $ such that $ D_{\mathbf{a}_{i,j}}(F) = b_{i,j} $, for $ j = 1,2,\ldots, r_i $ and $ i = 1,2, \ldots, n $.
\item
For all $ (\sigma,\delta) $-multiplicity sequences $ \mathbf{a}_i = (a_{i,1} , a_{i,2}, \ldots, a_{i,r_i}) \in \mathbb{F}^{r_i} $ such that $ a_{i,1} = a_i $, for $ i = 1,2, \ldots, n $, and setting $ N = \sum_{i=1}^n r_i $, the $ N \times N $ confluent Vandermonde matrix $ V_N(\mathbf{a}_1, \mathbf{a}_2, \ldots, \mathbf{a}_n) $ from Definition \ref{def confluent vandermonde} is invertible.
\end{enumerate}
\end{theorem}

\section{Multiplicity sequences in algebraic conjugacy classes} \label{sec multiplicity seqs}

In this section, we characterize $ (\sigma,\delta) $-multiplicity sequences (Definition \ref{def multiplicity seq}) over algebraic conjugacy classes assuming that $ \sigma $ is surjective, see Theorem \ref{th charact multiplicity seqs in alg conj}. Then, in Theorem \ref{th min skew pol of conjugacy class with mult}, we study what happens when all points in a single conjugacy class are zeros of the same multiplicity of a given skew polynomial.

We start with some auxiliary lemmas. The following result is \cite[Cor. 6.3]{lam-leroy-wedI}.

\begin{lemma} [\textbf{\cite{lam-leroy-wedI}}] \label{lemma algebraic conj class 1}
Let $ a \in \mathbb{F} $ and assume that its conjugacy class $ C^{\sigma,\delta}(a) \subseteq \mathbb{F} $ is algebraic, that is, $ I(C^{\sigma,\delta}(a)) \neq \{ 0 \} $ (Definition \ref{def algebraic sets}). If a skew polynomial $ F \in \mathbb{F}[x;\sigma,\delta] $ has $ a $ as a left zero (that is, $ x-a $ divides $ F $ on the left), then $ F $ has some element $ b \in C^{\sigma,\delta}(a) $ as a right zero (that is, $ x-b $ divides $ F $ on the right).
\end{lemma}
%

Another useful ingredient is the following characterization, which is \cite[Th. 5.10]{algebraic-conjugacy}.

\begin{lemma} [\textbf{\cite{algebraic-conjugacy}}] \label{lemma algebraic conj class 2}
For $ a \in \mathbb{F} $, the conjugacy class $ C^{\sigma,\delta}(a) \subseteq \mathbb{F} $ is algebraic if, and only if, $ \mathbb{F} $ has finite right dimension over $ K^{\sigma,\delta}_a $ (Definition \ref{def centralizer}). In such a case, $ C^{\sigma,\delta}(a) $ is P-closed and
$$ {\rm Rk}(C^{\sigma,\delta}(a)) = \dim^R_{K^{\sigma,\delta}_a}(\mathbb{F}) . $$
\end{lemma}
%
%
%

We will assume that $ \sigma $ is surjective due to the following result \cite[Eq. (19)]{ore}.

\begin{lemma} [\textbf{\cite{ore}}] \label{lemma ore left right isomorph}
If $ \sigma $ is surjective, then $ \mathbb{F}[x;\sigma,\delta] $ is a left Euclidean domain. In particular, if $ \sigma $ is surjective, then all definitions and results stated up to this point on the right hold in the same way on the left, and vice versa.
\end{lemma}
%

We will also use the notion of evaluation from \cite[Def. 3.1]{hilbert90} and \cite[Eq. (2.7)]{leroy-pol}.

\begin{definition} [\textbf{\cite{hilbert90, leroy-pol}}] \label{def lin op pols}
Given $ a \in \mathbb{F} $, we define the right $ K^{\sigma,\delta}_a $-linear map $ \mathcal{D}^{\sigma,\delta}_a : \mathbb{F} \longrightarrow \mathbb{F} $ such that $ \mathcal{D}^{\sigma,\delta}_a(\beta) = \sigma(\beta) a + \delta(\beta) $, for $ \beta \in \mathbb{F} $. We will denote $ \mathcal{D}_a = \mathcal{D}^{\sigma,\delta}_a $ if there is no confusion about the pair $ (\sigma,\delta) $. For a skew polynomial $ F = F_d x^d + \cdots + F_1 x + F_0 \in \mathbb{F}[x;\sigma,\delta] $, we define $ F^{\mathcal{D}_a} = F_d \mathcal{D}_a^d + \cdots + F_1 \mathcal{D}_a + F_0 {\rm Id} \in \mathbb{F}[\mathcal{D}_a] $.
\end{definition}

We will also use the following formula from \cite[Th. 2.8]{leroy-pol} (see also \cite[Lemma 24]{linearizedRS}).

\begin{lemma} [\textbf{\cite{leroy-pol}}] \label{lemma connecting evaluations}
For $ F \in \mathbb{F}[x;\sigma,\delta] $, $ a \in \mathbb{F} $ and $ \beta \in \mathbb{F}^* $, we have
$$ F ( a^\beta ) = F^{\mathcal{D}_a}(\beta) \beta^{-1}. $$
\end{lemma}

The following result is our main characterization of $ (\sigma,\delta) $-multiplicity sequences. It extends \cite[Lemma 4.1]{bolotnikov-zeros} (see also \cite[Th. 3.1]{bolotnikov-conf}) from quaternionic polynomials to general skew polynomials.

\begin{theorem} \label{th charact multiplicity seqs in alg conj}
Assume that $ \sigma $ is surjective and all conjugacy classes are algebraic. Let $ \mathbf{a} = (a_1, a_2, $ $ \ldots, a_r) \in \mathbb{F}^r $, where $ r \geq 2 $. The following are equivalent.
\begin{enumerate}
\item
$ \mathbf{a} $ is a $ (\sigma,\delta) $-multiplicity sequence, that is, $ a_1 $ is the only (right) zero of $ P_{\mathbf{a}} $.
\item
If $ b_1, \ldots, b_r \in \mathbb{F} $ satisfy $ P_{\mathbf{a}} = (x-b_r) \cdots (x-b_1) $, then $ b_i = a_i $, for $ i = 1,2, \ldots, r $.
\item
The pair $ (a_i, a_{i+1}) $ is a $ (\sigma,\delta) $-multiplicity sequence, for $ i = 1,2, \ldots, r-1 $.
\item
For $ i = 1,2, \ldots, r-1 $, there exists $ \beta_i \in \mathbb{F}^* $ such that $ a_{i+1} = a_i^{\beta_i} $ and
$$ \beta_i \notin \left\lbrace \left. \left( a_i^\beta - a_i \right) \beta \right| \beta \in \mathbb{F} \right\rbrace = \left\lbrace \left. \mathcal{D}_{a_i}^{\sigma,\delta}(\beta) - a_i \beta \right| \beta \in \mathbb{F} \right\rbrace. $$
\item
$ x - a_r $ is the only linear skew polynomial that divides $ P_{\mathbf{a}} $ on the left.
\end{enumerate}
\end{theorem}
\begin{proof}
We prove the following implications separately:

$ 3. \Longrightarrow 1. $: We proceed by induction in $ r \geq 2 $, being the case $ r = 2 $ trivial. Assume that item 3 holds for some $ r \geq 3 $. By induction hypothesis, $ (a_1, a_2, \ldots, a_{r-1}) \in \mathbb{F}^{r-1} $ is a $ (\sigma,\delta) $-multiplicity sequence. Define $ P = (x-a_{r-2}) \cdots (x-a_2) (x - a_1) $ and $ Q = (x-a_r) (x-a_{r-1}) $. Hence $ a_1 $ is the only zero of $ P $ by Proposition \ref{prop monotonicity multiplicity seqs}. Assume that $ \mathbf{a} = (a_1, a_2, \ldots, a_r) \in \mathbb{F}^r $ is not a $ (\sigma,\delta) $-multiplicity sequence. In other words, there exists $ b \in \mathbb{F} $ such that $ b \neq a_1 $ and $ P_{\mathbf{a}}(b) = 0 $. By the product rule (Lemma \ref{lemma product rule}), it holds that
$$ 0 = P_{\mathbf{a}}(b) = Q \left( b^{P(b)} \right) P(b), $$
where $ P(b) \neq 0 $ since $ b \neq a_1 $. Therefore we deduce that $ Q \left( b^{P(b)} \right) = 0 $. Since $ (a_{r-1},a_r) $ is a $ (\sigma,\delta) $-multiplicity sequence, we have that $ b^{P(b)} = a_{r-1} $. However, in that case, we have that $ R(b) = 0 $ by the product rule (Lemma \ref{lemma product rule}), for
$$ R = (x-a_{r-1}) P = (x-a_{r-1}) \cdots (x-a_2) (x - a_1), $$
which is absurd since $ b \neq a_1 $ but $ (a_1, a_2, \ldots, a_{r-1}) \in \mathbb{F}^{r-1} $ is a $ (\sigma,\delta) $-multiplicity sequence. Thus it must hold that $ \mathbf{a} \in \mathbb{F}^r $ is a $ (\sigma,\delta) $-multiplicity sequence, and we are done.

$ 2. \Longleftrightarrow 3. $: It is trivial that item 2 implies item 3. Assume now that item 3 holds. Since item 3 implies item 1, we deduce that $ (a_i, a_{i+1}, \ldots, a_r) \in \mathbb{F}^{r-i+1} $ is a $ (\sigma,\delta) $-multiplicity sequence, for $ i = 1,2, \ldots, r $. Assume that $ P_{\mathbf{a}} = (x-b_r) \cdots (x-b_2)(x-b_1) $ for $ b_1, b_2, \ldots, b_r \in \mathbb{F} $. Let $ j \in \mathbb{N} $ be the maximum number such that $ 1 \leq j \leq r $ and $ b_i = a_i $, for $ i = 1,2, \ldots, j $. If item 2 does not hold, then $ j \leq r-1 $. In that case, 
$$ (x - b_r) \cdots (x - b_{j+2}) (x - b_{j+1}) = (x - a_r) \cdots (x - a_{j+2}) (x - a_{j+1}). $$
However, since $ b_{j+1} \neq a_{j+1} $, then $ (a_{j+1}, a_{j+2}, \ldots, a_r) $ is not a $ (\sigma,\delta) $-multiplicity sequence. This is a contradiction, thus item 2 holds.

$ 1. \Longrightarrow 3. $: We proceed by induction on $ r \geq 2 $, being the case $ r = 2 $ trivial. Assume that item 1 holds for some $ r \geq 3 $. By induction hypothesis, $ (a_i, a_{i+1}) $ is a $ (\sigma,\delta) $-multiplicity sequence, for $ i = 1,2, \ldots, r-2 $. Assume however that $ (a_{r-1},a_r) $ is not a $ (\sigma,\delta) $-multiplicity sequence. By definition of minimal skew polynomial, we deduce that $ F_\Psi = (x-a_r) (x-a_{r-1}) $, where $ \Psi = \{ a_{r-1}, b \} \subseteq \mathbb{F} $ is P-independent and $ b \neq a_{r-1} $. Define now $ P = (x-a_{r-2}) \cdots (x-a_2) (x - a_1) $. Hence, we have that $ P_{\mathbf{a}} = F_\Psi P $.

By hypothesis and Proposition \ref{prop monotonicity multiplicity seqs}, $ (a_1, a_2, \ldots, a_{r-2}) $ is a $ (\sigma,\delta) $-multiplicity sequence, thus $ C^{\sigma,\delta}(a_1) \cap Z(P) = \{ a_1 \} $. Hence by Lemma \ref{lemma connecting evaluations}, $ {\rm Ker} \left( P^{\mathcal{D}_{a_1}} \right) = K^{\sigma,\delta}_{a_1} $. We also have that $ m = \dim^R_{K^{\sigma,\delta}_{a_1}} (\mathbb{F}) < \infty $ by Lemma \ref{lemma algebraic conj class 2}, since $ C^{\sigma,\delta}(a_1) $ is algebraic. Therefore
\begin{equation}
\dim^R_{K^{\sigma,\delta}_{a_1}} \left( {\rm Im} \left( P^{\mathcal{D}_{a_1}} \right) \right) = m - 1.
\label{eq proof equivalences multiplicity 2}
\end{equation}
Since $ C^{\sigma,\delta}(a_{r-1}) $ and $ C^{\sigma,\delta}(b) $ are algebraic and $ \mathbf{a} $ is a $ (\sigma,\delta) $-multiplicity sequence, we deduce from Proposition \ref{prop monotonicity multiplicity seqs} and Lemma \ref{lemma algebraic conj class 1} that there exist $ \alpha, \beta \in \mathbb{F}^* $ such that $ a_{r-1} = a_1^\alpha $ and $ b = a_1^\beta $. Next we prove that
\begin{equation}
 {\rm Im} \left( P^{\mathcal{D}_{a_1}} \right) \cap \langle \alpha, \beta \rangle^R_{K^{\sigma,\delta}_{a_1}} = \{ 0 \} ,
\label{eq proof equivalences multiplicity 1}
\end{equation}
where $ \langle \alpha, \beta \rangle^R_{K^{\sigma,\delta}_{a_1}} $ denotes the right $ K^{\sigma,\delta}_{a_1} $-linear span of $ \alpha $ and $ \beta $. Let $ \gamma \in {\rm Im} \left( P^{\mathcal{D}_{a_1}} \right) \cap \langle \alpha , \beta \rangle^R_{K^{\sigma,\delta}_{a_1}} $. Assume that $ \gamma \neq 0 $. Since $ \gamma \in {\rm Im} \left( P^{\mathcal{D}_{a_1}} \right) $, there exists $ \eta \in \mathbb{F}^* $ such that $ \gamma = P^{\mathcal{D}_{a_1}}(\eta) $. Observe that, by Lemma \ref{lemma connecting evaluations}, we have that $ P \left( a_1^\eta \right) = P^{\mathcal{D}_{a_1}}(\eta) \eta^{-1} = \gamma \eta^{-1} \neq 0 $. Now, since $ \gamma \in \langle \alpha, \beta \rangle^R_{K^{\sigma,\delta}_{a_1}} $, we deduce by Theorem \ref{th P-independence multiplicity seqs} that $ a_1^\gamma $ is P-dependent on $ a_{r-1} $ and $ b $. In other words, we have that
$$ 0 = F_\Psi \left( a_1^\gamma \right) = F_\Psi \left( a_1^{P \left( a_1^\eta \right) \eta } \right). $$
Hence, by the product rule (Lemma \ref{lemma product rule}), it holds that
$$ P_{\mathbf{a}} \left( a_1^\eta \right) = F_\Psi \left( a_1^{P \left( a_1^\eta \right) \eta } \right) P \left( a_1^\eta \right) = 0. $$
Since $ \mathbf{a} $ is a $ (\sigma,\delta) $-multiplicity sequence, we deduce that $ a_1^\eta = a_1 $, that is, $ \eta \in K^{\sigma,\delta}_{a_1} $. However, since $ {\rm Ker} \left( P^{\mathcal{D}_{a_1}} \right) = K^{\sigma,\delta}_{a_1} $, we would have that $ \gamma = P^{\mathcal{D}_{a_1}}(\eta) = 0 $, a contradiction. Thus (\ref{eq proof equivalences multiplicity 1}) must hold. Finally, since $ \Psi = \{ a_{r-1}, b \} $ is P-independent, we conclude by Theorem \ref{th P-independence multiplicity seqs} that $ \dim^R_{K^{\sigma,\delta}_{a_1}} \left( \langle \alpha, \beta \rangle^R_{K^{\sigma,\delta}_{a_1}} \right) = 2 $. Combined with (\ref{eq proof equivalences multiplicity 2}) and (\ref{eq proof equivalences multiplicity 1}), we have 
$$ \dim^R_{K^{\sigma,\delta}_{a_1}} \left( {\rm Im} \left( P^{\mathcal{D}_{a_1}} \right) + \langle \alpha, \beta \rangle^R_{K^{\sigma,\delta}_{a_1}} \right) = m+1. $$
This is absurd since $ \dim^R_{K^{\sigma,\delta}_{a_1}}(\mathbb{F}) = m < \infty $ by assumption. Therefore, $ (a_{r-1},a_r) $ must be a $ (\sigma,\delta) $-multiplicity sequence, and we are done.

$ 2. \Longleftrightarrow 5. $: It follows from Lemma \ref{lemma ore left right isomorph} and the equivalence between items 1 and 2, since item 2 can be read from left to right or from right to left.

$ 3. \Longleftrightarrow 4. $: First, if item 3 holds, then $ a_1, a_2, \ldots, a_r \in C^{\sigma,\delta}(a_1) $ by Lemma \ref{lemma algebraic conj class 1}, since all conjugacy classes are algebraic and $ (a_1, a_2, \ldots, a_i ) \in \mathbb{F}^i $ is a $ (\sigma,\delta) $-multiplicity sequence, for $ i = 1,2, \ldots, r $, by the fact that item 3 implies item 1. Thus, in both items 3 and 4 it holds that $ a_1, a_2, \ldots, a_r \in C^{\sigma,\delta}(a_1) $. 

Now fix $ i = 1,2, \ldots, r-1 $ and let $ \beta_i, \beta \in \mathbb{F}^* $ be such that $ a_{i+1} = a_i^{\beta_i} $ and $ a_i^\beta \neq a_i $, which is equivalent to $ \beta \in \mathbb{F} \setminus K^{\sigma,\delta}_{a_i} $. By the product rule (Lemma \ref{lemma product rule}),
$$ P ( a_i^\beta ) = \left( a_i^{ \left( a_i^\beta - a_i \right) \beta } - a_i^{\beta_i} \right) \left( a_i^\beta - a_i \right) $$
for $ P = (x-a_{i+1})(x-a_i) $. Thus it holds that $ P ( a_i^\beta ) = 0 $ for some $ \beta \in \mathbb{F}^* $ such that $ a_i^\beta \neq a_i $ if, and only if, $ a_i^{ \left( a_i^\beta - a_i \right) \beta } = a_i^{\beta_i} $, which is equivalent to
$$ \beta_i \in \left\langle \left\lbrace \left. \left( a_i^\beta - a_i \right) \beta \right| \beta \in \mathbb{F} \setminus K^{\sigma,\delta}_{a_i} \right\rbrace \right\rangle^R_{K^{\sigma,\delta}_{a_i}}. $$
We are done by the definition of $ \mathcal{D}_{a_i}^{\sigma,\delta} $ (Definition \ref{def lin op pols}), and by noticing that
$$ \left\langle \left\lbrace \left. \left( a_i^\beta - a_i \right) \beta \right| \beta \in \mathbb{F} \setminus K^{\sigma,\delta}_{a_i} \right\rbrace \right\rangle^R_{K^{\sigma,\delta}_{a_i}} = \left\lbrace \left. \left( a_i^\beta - a_i \right) \beta \right| \beta \in \mathbb{F} \right\rbrace . $$
\end{proof}

\begin{remark} \label{remark all conj classes are algebraic}
The assumption that all conjugacy classes are algebraic may look inconvenient. However, it holds in many classical particular cases (see Section \ref{sec particular cases}).
%
%
\end{remark}

A first consequence of Theorem \ref{th charact multiplicity seqs in alg conj} is that $ (\sigma,\delta) $-multiplicity sequences exist starting at any point of an algebraic conjugacy class. We will need the following lemma.

\begin{lemma} \label{lemma image of varphi is hyperplane}
If $ a \in \mathbb{F} $ and $ m = \dim^R_{K^{\sigma,\delta}_a}(\mathbb{F}) < \infty $, then $ \dim^R_{K^{\sigma,\delta}_a}(\mathcal{V}^{\sigma,\delta}_a) = m-1 $, where $ \mathcal{V}^{\sigma,\delta}_a = \{ \mathcal{D}^{\sigma,\delta}_a(\beta) - a \beta \mid \beta \in \mathbb{F} \} \subseteq \mathbb{F} $.
\end{lemma}
\begin{proof}
The map $ \varphi^{\sigma,\delta}_a : \mathbb{F} \longrightarrow \mathbb{F} $ given by $ \varphi^{\sigma,\delta}_a(\beta) = \mathcal{D}^{\sigma,\delta}_a(\beta) - a \beta $, for all $ \beta \in \mathbb{F} $, is right $ K^{\sigma,\delta}_a $-linear and satisfies that $ {\rm Ker} (\varphi^{\sigma,\delta}_a) = K^{\sigma,\delta}_a $ and $ {\rm Im}(\varphi^{\sigma,\delta}_a) = \mathcal{V}^{\sigma,\delta}_a $.
\end{proof}

We now prove the existence of $ (\sigma,\delta) $-multiplicity sequences over algebraic conjugacy classes. This result follows from Lemma \ref{lemma algebraic conj class 2}, item 4 in Theorem \ref{th charact multiplicity seqs in alg conj} and Lemma \ref{lemma image of varphi is hyperplane}.

\begin{corollary} \label{cor existence multiplicity seqs}
For all $ a \in \mathbb{F} $ such that $ \dim^R_{K^{\sigma,\delta}_a}(\mathbb{F}) < \infty $, and for all $ r \in \mathbb{Z}_+ $, there exists a $ (\sigma,\delta) $-multiplicity sequence $ \mathbf{a} = (a_1, a_2, \ldots, a_r) \in \mathbb{F}^r $ such that $ a_1 = a $.
\end{corollary}

On the other hand, we do not have in general uniqueness of $ (\sigma,\delta) $-multiplicity sequences starting at a point in an algebraic conjugacy class. However, there is a unique way to consider multiplicities as in Definition \ref{def multiplicity II} over a full conjugacy class, as the following theorem shows. We will assume again that all conjugacy classes are algebraic. We also need an auxiliary lemma, obtained from \cite[Lemma 5.2]{algebraic-conjugacy} or \cite[Prop. 18]{multivariateskew}.

\begin{lemma} [\textbf{\cite{algebraic-conjugacy, multivariateskew}}] \label{lemma ideal of conjugacy twosided}
Given $ a \in \mathbb{F} $, the left ideal $ I(C^{\sigma,\delta}(a)) $ is two-sided. 
\end{lemma}

\begin{theorem} \label{th min skew pol of conjugacy class with mult}
Let $ a \in \mathbb{F} $ and $ r \in \mathbb{Z}_+ $, and assume that all conjugacy classes are algebraic. Let $ a_1, a_2, \ldots, a_m \in C^{\sigma,\delta}(a) $ form a P-basis of $ C^{\sigma,\delta}(a) $. The following hold:
\begin{enumerate}
\item
If $ \mathbf{b} = (b_1, b_2, \ldots, b_r) \in \mathbb{F}^r $ is a $ (\sigma,\delta) $-multiplicity sequence such that $ b_1 \in C^{\sigma,\delta}(a) $, then $ P_\mathbf{b} $ divides $ F_{C^{\sigma,\delta}(a)}^r $ on the right.
\item
It holds that $ F_\Omega = F_{C^{\sigma,\delta}(a)}^r $, where $ \Omega = \{ P_{\mathbf{a}_1}, P_{\mathbf{a}_2}, \ldots, P_{\mathbf{a}_m} \} $ and $ \mathbf{a}_i = (a_{i,1}, a_{i,2}, \ldots, $ $ a_{i,r}) \in \mathbb{F}^r $ are $ (\sigma,\delta) $-multiplicity sequences with $ a_{i,1} = a_i $, for $ i = 1,2, \ldots, m $.
\item
The subset of $ \mathcal{U}_2 $ from Equation (\ref{eq universal set II}) given by
$$ C^{\sigma,\delta}_r(a) = \{ P_{\mathbf{b}} \mid \mathbf{b} \in \mathbb{F}^i, (\sigma,\delta) \textrm{-multiplicity seq.}, b_1 \in C^{\sigma,\delta}(a), 1 \leq i \leq r \} $$
is P-closed with minimal skew polynomial $ F_{C^{\sigma,\delta}(a)}^r $, P-basis $ \Omega $ and $ {\rm Rk}(C^{\sigma,\delta}_r(a)) = r \cdot {\rm Rk}(C^{\sigma,\delta}(a)) $. It may be seen as the conjugacy class of $ a $ with multiplicity $ r $.
\end{enumerate}
\end{theorem}
\begin{proof}
Denote $ P = F_{C^{\sigma,\delta}(a)} $ for simplicity. We start by proving item 1. We will show by induction in $ r \in \mathbb{Z}_+ $ that $ P_\mathbf{b} $ divides $ P^r $ on the right. The case $ r = 1 $ is trivial by definition. Consider now that $ r \geq 2 $. Since $ b_1 \in C^{\sigma,\delta}(a) $, there exists $ G \in \mathbb{F}[x;\sigma,\delta] $ such that $ P = G(x-b_1) $. By Lemma \ref{lemma ideal of conjugacy twosided}, there exists $ H \in \mathbb{F}[x;\sigma,\delta] $ such that 
\begin{equation}
P^r = P^{r-1} G (x-b_1) = H P^{r-1} (x-b_1).
\label{eq proof multiplicities on full conjugacy}
\end{equation}
By Theorem \ref{th charact multiplicity seqs in alg conj}, $ \mathbf{b}^\prime = (b_2, b_3, \ldots, b_r) \in \mathbb{F}^{r-1} $ is also a $ (\sigma,\delta) $-multiplicity sequence such that $ b_2 \in C^{\sigma,\delta}(a) $. Thus, by induction hypothesis, $ P_{\mathbf{b}^\prime} $ divides $ P^{r-1} $ on the right. Hence $ P_\mathbf{b} = P_{\mathbf{b}^\prime} (x-b_1) $ divides $ P^r $ on the right by (\ref{eq proof multiplicities on full conjugacy}).

Now, we prove item 2. Since $ \Omega $ is P-independent by Theorem \ref{th P-independence multiplicity seqs}, we deduce that $ \deg(F_\Omega) = \deg(P^r) = mr $. Moreover, as $ P_{\mathbf{a}_i} $ divides $ P^r $ on the right, for $ i = 1,2, \ldots, m $, we conclude that $ F_\Omega $ divides $ P^r $ on the right. Since they have the same degree, then $ F_\Omega = P^r $ and we are done.

Finally, we prove item 3. With the results obtained above, combined with the product rule (Lemma \ref{lemma product rule}) and Proposition \ref{prop properties alg geom}, it is easy to see that $ I(C^{\sigma,\delta}_r(a)) = (F_{C^{\sigma,\delta}(a)}^r) $ and $ Z_{\mathcal{U}_2}(I(C^{\sigma,\delta}_r(a))) = Z_{\mathcal{U}_2}(\{ F_{C^{\sigma,\delta}(a)}^r \}) = C^{\sigma,\delta}_r(a) $, for $ \mathcal{U}_2 $ as in Equation (\ref{eq universal set II}). Therefore, $ C^{\sigma,\delta}_r(a) $ is P-closed with minimal skew polynomial $ F_{C^{\sigma,\delta}(a)}^r $. Furthermore, $ \Omega $ is P-independent by Theorem \ref{th P-independence multiplicity seqs}, and it is a set of P-generators of $ C^{\sigma,\delta}_r(a) $ by Lemma \ref{lemma P-generators iff min skew pols equal}, since $ F_\Omega = F_{C^{\sigma,\delta}(a)}^r $. Thus, $ \Omega $ is a P-basis of $ C^{\sigma,\delta}_r(a) $.
\end{proof}

It follows that, when considering a full conjugacy class, multiplicities as in Definition \ref{def multiplicity II} work as expected and do not depend on $ (\sigma,\delta) $-multiplicity sequences.

\begin{corollary} \label{cor multiplicities on full conjugacy}
Let the notation and assumptions be as in Theorem \ref{th min skew pol of conjugacy class with mult}. Given $ F \in \mathbb{F}[x;\sigma,\delta] $, and with multiplicities as in Definition \ref{def multiplicity II}, the following are equivalent:
\begin{enumerate}
\item
Every element in $ C^{\sigma,\delta}(a) $ is a zero of $ F $ of multiplicity $ r $.
\item
Every element in a P-basis of $ C^{\sigma,\delta}(a) $ is a zero of $ F $ of multiplicity $ r $.
\item
$ F_{C^{\sigma,\delta}(a)}^r $ divides $ F $ (on the right, on the left or in between, by Lemma \ref{lemma ideal of conjugacy twosided}).
\end{enumerate}
\end{corollary}
%
%

It is natural to ask if we could have arrived at Corollary \ref{cor multiplicities on full conjugacy} from the notion of multiplicity in Definition \ref{def multiplicity I}. We now show that this is not the case. 

\begin{example} \label{ex full conjugacy class pathology for naive mult}
Consider the setting of Example \ref{ex power of linear term is no multiplicity}, and further assume that $ m = 2 $. Using that $ a^{q-1} = -1 $, the reader may verify that $ a = \gamma^{ \frac{q+1}{2} } $ and $ b = - \gamma^{ \frac{q+1}{2} } $ are conjugate. Observe also that $ \{ a,b \} $ is P-independent, $ K^{\sigma,\delta}_a = \mathbb{F}_q $ and, by Lemma \ref{lemma algebraic conj class 2}, $ {\rm Rk}(C^{\sigma,\delta}(a)) = \dim_{\mathbb{F}_q}(\mathbb{F}_{q^m}) = m = 2 $. Therefore, $ \{ a,b \} $ forms a P-basis of $ C^{\sigma,\delta}(a) $. However, 
$$ F_{C^{\sigma,\delta}(a)} = (x-a)^2 = (x-b)^2 = x^2 + \gamma^{q+1} \in \mathbb{F}_q[x^2] . $$
Hence $ F_{C^{\sigma,\delta}(a)} $ vanishes at a P-basis of $ C^{\sigma,\delta}(a) $ with multiplicity $ 2 $ according to Definition \ref{def multiplicity I}. Clearly $ F_{C^{\sigma,\delta}(a)}^2 $ does not divide $ F_{C^{\sigma,\delta}(a)} $, and $ \deg(F_{C^{\sigma,\delta}(a)}) = 2 $, hence Corollary \ref{cor multiplicities on full conjugacy} and Theorem \ref{th upper bound multiplicities by degree II} do not hold for multiplicities as in Definition \ref{def multiplicity I}. 
\end{example}

\section{Particular cases and worked examples} \label{sec particular cases}

In this section, we study particular cases for which all conjugacy classes are algebraic, hence the results in the previous section hold. For worked examples, see \cite{skew-multi}.

We start with Bolotnikov's characterization of multiplicity sequences of quaternions \cite{bolotnikov-conf, bolotnikov-zeros} (called \textit{spherical chains} there). Quaternionic Lagrange interpolation has recently been motivated by the study of tight frames \cite{waldron-interpol, waldron-frames}. Denote by $ \mathbb{R} $ the real field. The division ring of quaternions is given by 
$$ \mathbb{H} = \{ a + b \mathbf{i} + c \mathbf{j} + d \mathbf{k} \mid a,b,c,d \in \mathbb{R} \}, $$ 
where $ \mathbf{i} $, $ \mathbf{j} $ and $ \mathbf{k} $ commute with real numbers and $ \mathbf{i}^2 = \mathbf{j}^2 = \mathbf{k}^2 = \mathbf{i} \mathbf{j} \mathbf{k} = - 1 $. Given a quaternion $ q = a + b \mathbf{i} + c \mathbf{j} + d \mathbf{k} \in \mathbb{H} $, where $ a,b,c,d \in \mathbb{R} $, the real and imaginary parts of $ q $ are given, respectively, by $ {\rm Re}(q) = a $ and $ {\rm Im}(q) = b \mathbf{i} + c \mathbf{j} + d \mathbf{k} $. The quaternionic conjugate of $ q \in \mathbb{H} $ is given by $ \overline{q} = {\rm Re}(q) - {\rm Im}(q) \in \mathbb{H} $, whereas its quaternionic modulus is given by $ |q| = \sqrt{q \overline{q}} $. Two quaternions $ q, t \in \mathbb{H} $ are conjugate, in the sense of division rings, if there exists $ u \in \mathbb{H}^* $ such that $ t = u q u^{-1} $. This is actually equivalent to $ {\rm Re} (q) = {\rm Re} (t) $ and $ |q| = |t| $ (see \cite[p. 333]{brenner}). The following result is a consequence of Theorem \ref{th charact multiplicity seqs in alg conj} and was obtained in \cite[Lemma 4.1]{bolotnikov-zeros} (see also \cite[Th. 3.1]{bolotnikov-conf}).

\begin{corollary} [\textbf{Quaternions \cite{bolotnikov-zeros}}] \label{cor multiplicity over quaternions}
Let $ \mathbb{F} = \mathbb{H} $ and $ (\sigma, \delta) = ({\rm Id}, 0) $. The vector $ (a_1, a_2, $ $ \ldots, $ $ a_r) \in \mathbb{H}^r $ is a $ (\sigma,\delta) $-multiplicity sequence if, and only if, either $ a_1 = a_2 = \ldots = a_r \in \mathbb{R} $, or $ a_1, a_2, \ldots, a_r \in \mathbb{H} \setminus \mathbb{R} $ and for $ i = 1,2, \ldots, r-1 $, it holds that
$$ {\rm Re}(a_{i+1}) = {\rm Re}(a_i), \quad |a_{i+1}| = |a_i|, \quad \textrm{and} \quad \overline{a}_{i+1} \neq a_i . $$
In particular, $ (a,\ldots,a) \in \mathbb{H}^r $ is a $ (\sigma,\delta) $-multiplicity sequence, for $ a \in \mathbb{H} $ and $ r \in \mathbb{Z}_+ $.
\end{corollary}
\begin{proof}
First, since $ \dim_{\mathbb{R}}(\mathbb{H}) = 4 $, all conjugacy classes are algebraic by Lemma \ref{lemma algebraic conj class 2}. Second, if $ a \in \mathbb{R} $, then $ C^{\sigma,\delta}(a) = \{ a \} $, and if $ a \in \mathbb{H} \setminus \mathbb{R} $, then $ C^{\sigma,\delta}(a) \cap \mathbb{R} = \varnothing $. From these facts and Theorem \ref{th charact multiplicity seqs in alg conj}, if $ a,b \in \mathbb{H} $ are conjugate in the sense of division rings, then $ (a,b) \in \mathbb{H}^2 $ is not a $ (\sigma,\delta) $-multiplicity sequence if, and only if, $ a,b \in \mathbb{H} \setminus \mathbb{R} $ and $ b = \overline{a} $, and the characterization follows. 
\end{proof}

We now illustrate Corollary \ref{cor multiplicities on full conjugacy} for $ \mathbb{F} = \mathbb{H} $ and $ (\sigma,\delta) = ({\rm Id},0) $. Recall that, given $ a \in \mathbb{H} $, its conjugacy class is $ C^{\sigma,\delta}(a) = \{ a \} $ if $ a \in \mathbb{R} $, and $ C^{\sigma,\delta}(a) = \{ b \in \mathbb{H} \mid {\rm Re}(b) = {\rm Re}(a), |b| = |a| \} $ if $ a \in \mathbb{H} \setminus \mathbb{R} $. In the case $ a \in \mathbb{H} \setminus \mathbb{R} $, this is due to the fact that
$$ F_{C^{\sigma,\delta}(a)} = (x-a) (x - \overline{a}) = x^2 -2 {\rm Re}(a) x + |a|^2  \in \mathbb{R}[x]. $$
Thus the following result is immediate from Corollary \ref{cor multiplicities on full conjugacy}.

\begin{corollary}
Consider $ \mathbb{F} = \mathbb{H} $, $ (\sigma,\delta) = ({\rm Id},0) $, $ a \in \mathbb{H} $, $ F \in \mathbb{H}[x] $ and $ r \in \mathbb{Z}_+ $. 
\begin{enumerate}
\item
If $ a \in \mathbb{R} $, then $ a $ is a zero of $ F $ of multiplicity $ r $ (according to Definition \ref{def multiplicity II}) if, and only if, the central polynomial $ (x-a)^r \in \mathbb{R}[x] $ divides $ F $.
\item
If $ a \in \mathbb{H} \setminus \mathbb{R} $, then $ a $ and $ \overline{a} $ are zeros of $ F $, each of multiplicity $ r $ (according to Definition \ref{def multiplicity II}) if, and only if, $ F $ is divisible by the central polynomial
$$ (x-a)^r (x - \overline{a})^r = \left( x^2 -2 {\rm Re}(a) x + |a|^2 \right)^r \in \mathbb{R}[x]. $$
\end{enumerate}
\end{corollary}

We now turn to cyclic Galois extensions of fields. This includes the case of skew polynomials over finite fields, whose evaluations give projective polynomials \cite{projective-pols} (Definition \ref{def evaluation high degree}) and linearized polynomials \cite{orespecial} (Definition \ref{def lin op pols}). This case is of interest in Coding Theory. For instance, linearized Reed-Solomon codes \cite{linearizedRS} provide PMDS codes with the smallest field sizes known so far \cite{cai-field, gopi-field, pmds-conjugacy, universal-lrc}. \textit{Multiplicity enhancements} are of interest in similar areas of Computer Science \cite{dvir}.

If $ \mathbb{F} $ is a field and $ \sigma $ has finite order $ m \in \mathbb{Z}_+ $, then for each $ a \in \mathbb{F}^* $, it holds that $ K^{\sigma,\delta}_a = K = \{ \beta \in \mathbb{F} \mid \sigma(\beta) = \beta \} $, and $ K \subseteq \mathbb{F} $ is a Galois extension with cyclic Galois group $ \mathcal{G} = \{ {\rm Id}, \sigma, \sigma^2, \ldots, \sigma^{m-1} \} $ of order $ m \in \mathbb{Z}_+ $. In particular, all conjugacy classes are algebraic by Lemma \ref{lemma algebraic conj class 2}. In this setting, we may assume without loss of generality that $ \delta = 0 $ (any other derivation would be inner for some $ c \in \mathbb{F} $, and the results are simply translated by $ c $, see \cite[Sec. 8.3]{cohn} or \cite[Sec. 4.1]{linearizedRS}). We will denote by $ {\rm T}_{\mathbb{F}/K}, {\rm N}_{\mathbb{F}/K} : \mathbb{F} \longrightarrow K $ the classical trace and norm maps, given for $ a \in \mathbb{F} $ by
$$ {\rm T}_{\mathbb{F}/K}(a) = \sum_{i=0}^{m-1} \sigma^i(a) \quad \textrm{and} \quad {\rm N}_{\mathbb{F}/K}(a) = \prod_{i=0}^{m-1} \sigma^i(a). $$

\begin{corollary} [\textbf{Cyclic Galois extensions}] \label{cor multiplicity over cyclic galois}
Assume that $ \mathbb{F} $ is a field, $ \sigma $ has finite order $ m \in \mathbb{Z}_+ $, let $ \delta = 0 $ and define $ K = \{ \beta \in \mathbb{F} \mid \sigma(\beta) = \beta \} \subseteq \mathbb{F} $. The vector $ \mathbf{a} = (a_1, a_2, \ldots, a_r) \in \mathbb{F}^r $ is a $ (\sigma,\delta) $-multiplicity sequence if, and only if, either $ a_1 = a_2 = \ldots = a_r = 0 $ or $ a_1 \neq 0 $ and, for $ i = 1,2, \ldots, r-1 $, there exists $ \beta_i \in \mathbb{F}^* $ such that
$$ a_{i+1} = \sigma(\beta_i)\beta_i^{-1} a_i \quad \textrm{and} \quad {\rm T}_{\mathbb{F}/K} \left( \beta_i a_i^{-1} \right) \neq 0. $$
In particular, $ (a,\ldots, a) $ is a $ (\sigma,\delta) $-multiplicity sequence if, and only if, $ {\rm T}_{\mathbb{F}/K} \left( a^{-1} \right) \neq 0 $.
\end{corollary}
\begin{proof}
First, all conjugacy classes are algebraic by Lemma \ref{lemma algebraic conj class 2}. By Theorem \ref{th charact multiplicity seqs in alg conj}, $ \mathbf{a} $ is a $ (\sigma,\delta) $-multiplicity sequence if, and only if, there exists $ \beta_i \in \mathbb{F}^* $ such that $ a_{i+1} = \sigma(\beta_i)\beta_i^{-1} a_i $ and $ \beta_i \notin \left\lbrace \left( \sigma(\beta) - \beta \right)a_i \mid \beta \in \mathbb{F} \right\rbrace $, for $ i = 1,2, \ldots, r-1 $. For a given $ i = 1,2,\ldots, r-1 $, the latter condition on $ \beta_i \in \mathbb{F}^* $ does not hold if, and only if, $ {\rm Tr}_{\mathbb{F}/K} \left( \beta_i a_i^{-1} \right) = 0 $, by Hilbert's additive Theorem 90 \cite[p. 290, Th. 6.3]{lang}, and the first statement follows. Second, if $ (a,\ldots,a) \in \mathbb{F}^r $, then $ \beta \in \mathbb{F}^* $ satisfies that $ a = \sigma(\beta)\beta^{-1} a $ if, and only if, $ \beta \in K $. In that case, $ {\rm T}_{\mathbb{F}/K}(\beta a^{-1}) = \beta {\rm T}_{\mathbb{F}/K}(a^{-1}) $, and the last statement follows. 
\end{proof}
%

We may now rephrase Corollary \ref{cor multiplicities on full conjugacy} for cyclic Galois extensions of fields. In the setting of Corollary \ref{cor multiplicity over cyclic galois} above, Hilbert's Theorem 90 \cite[p. 288, Th. 6.1]{lang} is equivalent to saying that the conjugacy class of $ a \in \mathbb{F}^* $ is an algebraic P-closed set with minimal skew polynomial $ F_{C^{\sigma,\delta}(a)} = x^m - N_{\mathbb{F}/K}(a) \in K[x^m] $ (see \cite[Th. 19]{lam}). Thus the following consequence of Corollary \ref{cor multiplicities on full conjugacy} can be seen as a Hilbert Theorem 90 with multiplicities.

\begin{corollary} [\textbf{Multiplicity Hilbert 90}] \label{cor hilbert 90}
Assume that $ \mathbb{F} $ is a field, $ \sigma $ has finite order $ m \in \mathbb{Z}_+ $, $ \delta = 0 $ and define $ K = \{ \beta \in \mathbb{F} \mid \sigma(\beta) = \beta \} \subseteq \mathbb{F} $. Given $ F \in \mathbb{F}[x;\sigma,\delta] $ and $ a \in \mathbb{F}^* $, and with multiplicities as in Definition \ref{def multiplicity II}, the following are equivalent:
\begin{enumerate}
\item
Every element in $ C^{\sigma,\delta}(a) $ is a zero of $ F $ of multiplicity $ r $.
\item
Every element in a P-basis of $ C^{\sigma,\delta}(a) $ is a zero of $ F $ of multiplicity $ r $.
\item
The central skew polynomial $ \left( x^m - N_{\mathbb{F}/K}(a) \right)^r \in K[x^m] $ divides $ F $.
\end{enumerate}
In particular, if $ a_1, a_2, \ldots, a_\ell \in \mathbb{F}^* $ are pair-wise non-conjugate, $ F \neq 0 $, and every element in $ C^{\sigma,\delta}(a_i) $ is a zero of $ F $ of multiplicity $ r_i \in \mathbb{Z}_+ $, for $ i = 1,2, \ldots, \ell $, then
$$ m \sum_{i=1}^\ell r_i \leq \deg(F) . $$
Furthermore, equality holds if, and only if, $ F = b \prod_{i=1}^\ell \left( x^m - N_{\mathbb{F}/K}(a_i) \right)^{r_i} $, where $ b \in \mathbb{F}^* $.
\end{corollary}

A particular case of interest is that of skew polynomials over finite fields (note that Examples \ref{ex power of linear term is no multiplicity} and \ref{ex full conjugacy class pathology for naive mult} above are of this form).

\begin{example} [\textbf{Finite fields}] \label{ex multiplicity over finite fields}
Let $ q $ be a power of a prime number and let $ m \in \mathbb{Z}_+ $. Let $ \mathbb{F}_q \subseteq \mathbb{F}_{q^m} $ be an extension of finite fields of sizes $ q $ and $ q^m $, respectively. The Galois group of this extension is cyclic and generated by the $ q $-Frobenius automorphism $ \sigma : \mathbb{F}_{q^m} \longrightarrow \mathbb{F}_{q^m} $, given by $ \sigma (a) = a^q $, for all $ a \in \mathbb{F}_{q^m} $. As above, we may assume that $ \delta = 0 $. 

In this setting, Corollary \ref{cor multiplicity over cyclic galois} says that $ \mathbf{a} = (a_1, a_2, \ldots, a_r) \in \mathbb{F}_{q^m}^r $ is a $ (\sigma,\delta) $-multiplicity sequence if, and only if, $ a_1 = a_2 = \ldots = a_r = 0 $ or $ a_1 \neq 0 $ and there exist $ \beta_i \in \mathbb{F}_{q^m}^* $ such that
$$ a_{i+1} = \beta_i^{q-1} a_i \quad \textrm{and} \quad {\rm T}_{\mathbb{F}_{q^m} / \mathbb{F}_q} \left( \beta_i a_i^{-1} \right) = \sum_{j=0}^{m-1} \beta_i^{q^j} a_i^{-q^j} \neq 0, $$
for $ i = 1,2, \ldots, r-1 $. In particular, if $ a_1 \neq 0 $, then we have that 
$$ {\rm N}_{\mathbb{F}_{q^m} / \mathbb{F}_q} \left( a_{i+1} a_i^{-1} \right) = \prod_{j=0}^{m-1} a_{i+1}^{q^j} a_i^{-q^j} = 1. $$
Observe that we may replace $ \sigma $ by any other generator of the Galois group, that is, we may consider $ \sigma(a) = a^{q^s} $, for all $ a \in \mathbb{F}_{q^m} $, where $ s \in \mathbb{Z}_+ $ is coprime with $ m $.

Now, consider the particular setting of Examples \ref{ex power of linear term is no multiplicity} and \ref{ex full conjugacy class pathology for naive mult}. That is, consider $ q $ odd, $ m $ even with $ m \geq 2 $, and define
$$ a = \gamma^{ \frac{1}{2} \cdot \frac{q^m - 1}{q-1}} \quad \textrm{and} \quad b = - \gamma^{ \frac{1}{2} \cdot \frac{q^m - 1}{q-1}} . $$
where $ \gamma \in \mathbb{F}_{q^m}^* $ is a primitive element. According to Example \ref{ex power of linear term is no multiplicity}, $ (a,a) \in \mathbb{F}_{q^m}^* $ is not a $ (\sigma,\delta) $-multiplicity sequence since a simple computation shows that $ (x-a)^2 = (x-b)^2 $ and $ a \neq b $. We may verify this again by checking the conditions in Corollary \ref{cor multiplicity over cyclic galois}. We simply note that, since $ a^q = -a $, then
$$ {\rm T}_{\mathbb{F}_{q^m}/\mathbb{F}_q} \left( a^{-1} \right) = \frac{m}{2} \cdot a^{-1} - \frac{m}{2} \cdot a^{-1} = 0, $$
since $ m $ is even. Thus Corollary \ref{cor multiplicity over cyclic galois} also states that $ (a,a) $ is not a $ (\sigma,\delta) $-multiplicity sequence.
\end{example}

Another interesting example is that of complex conjugation.

\begin{example} [\textbf{Complex conjugation}]  \label{ex multiplicity over complex conj}
Consider the cyclic Galois extension $ \mathbb{R} \subseteq \mathbb{C} $ of order $ 2 $, where $ \mathbb{C} $ denotes the complex field. As is well known, the generator of the Galois group of this extension is $ \sigma : \mathbb{C} \longrightarrow \mathbb{C} $, given by complex conjugation, that is,
$$ \sigma(a) = \overline{a} = {\rm Re}(a) - {\rm Im}(a), $$
for all $ a \in \mathbb{C} $. As above, we may assume that $ \delta = 0 $.

In this setting, Corollary \ref{cor multiplicity over cyclic galois} says that $ \mathbf{a} = (a_1, a_2, \ldots, a_r) \in \mathbb{C}^r $ is a $ (\sigma,\delta) $-multiplicity sequence if, and only if, $ a_1 = a_2 = \ldots = a_r = 0 $ or $ a_1 \neq 0 $ and there exists $ \beta_i \in \mathbb{C}^* $ such that
$$ a_{i+1} = \overline{\beta}_i \beta_i^{-1} a_i \quad \textrm{and} \quad {\rm T}_{\mathbb{C} / \mathbb{R}} \left( \beta_i a_i^{-1} \right) = \beta_i a_i^{-1} + \overline{\beta_i a_i^{-1}} \neq 0, $$
for $ i = 1,2, \ldots, r-1 $. By a straightforward calculation we deduce that, if $ a_1 \neq 0 $, then $ \mathbf{a} $ is a $ (\sigma,\delta) $-multiplicity sequence if, and only if, 
$$ |a_{i+1}| = | a_i | \quad \textrm{and} \quad \overline{a}_{i+1} \neq - a_i, $$
for $ i = 1,2, \ldots, r-1 $. It is interesting to note that $ \mathbf{a} = (a,a, \ldots, a) \in \mathbb{C}^r $ is a $ (\sigma,\delta) $-multiplicity sequence if, and only if, $ a $ is not purely imaginary, that is, $ \overline{a} \neq -a $. We may verify the reversed implication directly by noting that, if $ a = ib \in \mathbb{C} $, where $ b \in \mathbb{R}^* $ and $ i $ denotes the imaginary unit, then
$$ (x-a)^2 = x^2 - b^2 \in \mathbb{R}[x^2], $$
operating inside the ring $ \mathbb{C}[x;\sigma,\delta] $. Therefore, $ (x-a)^2 $ has two distinct zeros, namely, $ b $ and $ -b $. We leave to the reader the verification that $ (x-a)^2 \in \mathbb{C}[x;\sigma,\delta] $ only has one zero (without counting multiplicities) if $ a $ is not purely imaginary.
\end{example}

Another common scenario is that of standard derivations over fields, which appears in \textit{differential Galois Theory} \cite{jacobson-derivation, singer}. In this case, the corresponding skew polynomial ring is commonly called a \textit{differential polynomial ring}. The following result is straightforward by item 4 in Theorem \ref{th charact multiplicity seqs in alg conj}. Again in this case, all conjugacy classes are algebraic.

\begin{corollary} [\textbf{Standard derivations}]  \label{cor multiplicity over derivations}
Assume that $ \mathbb{F} $ is a field, $ \sigma = {\rm Id} $ is the identity morphism, and $ \delta $ is an arbitrary $ {\rm Id} $-derivation (that is, a standard derivation of the field $ \mathbb{F} $). Let $ K = \{ \beta \in \mathbb{F} \mid \delta(\beta) = 0 \} \subseteq \mathbb{F} $ denote the subfield of constants of $ \delta $, and assume that $ \dim_K(\mathbb{F}) < \infty $. The vector $ (a_1, a_2, \ldots, a_r) \in \mathbb{F}^r $ is a $ (\sigma,\delta) $-multiplicity sequence if, and only if, there exist $ \beta_i \in \mathbb{F}^* $ such that, for $ i = 1,2, \ldots, r-1 $,
$$ a_{i+1} - a_i = \delta(\beta_i) \beta_i^{-1} \quad \textrm{and} \quad \beta_i \notin \delta(\mathbb{F}). $$
\end{corollary}

In the setting of Corollary \ref{cor multiplicity over derivations} above, the dimension of $ \mathbb{F} $ over $ K $ can be determined from the minimal polynomial of the derivation $ \delta $. The following result is a particular case of \cite[Th. 5.10]{algebraic-conjugacy}.

\begin{lemma}[\textbf{\cite{algebraic-conjugacy}}] \label{lemma algebraic derivations}
Assume that $ \mathbb{F} $ is a field and let $ \delta $ be an arbitrary derivation. If $ K = \{ \beta \in \mathbb{F} \mid \delta(\beta) = 0 \} \subseteq \mathbb{F} $, then $ \mathbb{F} $ has finite dimension over $ K $ if, and only if, there exists a non-zero polynomial $ P \in \mathbb{F}[y] $ such that $ P(\delta) = 0 $, that is, $ \delta $ is algebraic. In such a case, $ \dim_K(\mathbb{F}) $ is the degree of the minimum polynomial of $ \delta $ over $ \mathbb{F} $.
\end{lemma}

In particular, in positive characteristic, the \textit{D-fields} introduced by Jacobson in \cite[Sec. III]{jacobson-derivation} fall into the setting of Corollary \ref{cor multiplicity over derivations}. Perhaps the simplest non-trivial examples of D-fields are given by rational functions in positive characteristic.

\begin{example} [\textbf{Rational function fields}] \label{ex multiplicity over rational function fields}
Let $ k $ be a field of characteristic $ p > 0 $. Define $ \mathbb{F} = k(z) $ to be the field of rational functions over $ k $ in $ z $. Let $ \delta : \mathbb{F} \longrightarrow \mathbb{F} $ be a non-zero derivation (note that they are all given by $ \delta = c d/dz $, where $ c \in \mathbb{F}^* $ and $ d/dz $ is the usual derivation). As noted by Jacobson \cite{jacobson-derivation}, since $ p > 0 $, it holds that $ \delta^p $ is again a derivation. Furthermore, it is trivial to see that $ \delta^p = 0 $ in $ k[z] $. Using the defining properties of derivations, we deduce that
$$ \delta^p (f g^{-1}) = f \delta^p(g^{-1}) + \delta^p(f) g^{-1} = - f \delta^p(g) g^{-2} = 0, $$
for all $ f, g \in k[z] $ with $ g \neq 0 $. Hence we deduce that $ \delta^p = 0 $ in $ \mathbb{F} $. In fact, $ P = y^p - 1 \in \mathbb{F}[y] $ is the minimal polynomial of $ \delta $ over $ \mathbb{F} $. We may check the validity of Lemma \ref{lemma algebraic derivations} by noting that
$$ K = \{ \beta \in \mathbb{F} \mid \delta(\beta) = 0 \} = k(z^p), $$
and $ 1,z,z^2, \ldots, z^{p-1} \in \mathbb{F} $ form a basis of $ \mathbb{F} $ as a vector space over $ K $ (the proof is straightforward and left to the reader). Now, by writing an element in $ \mathbb{F} $ as $ f(z) = \sum_{i=0}^{p-1} f_i(z^p) z^i $, where $ f_i(z^p) \in K $, for $ i = 0,1, \ldots, p-1 $, it is easy to see that
\begin{equation}
\delta (\mathbb{F}) = \langle 1,z, z^2, \ldots, z^{p-2} \rangle_K .
\label{eq example rational functions}
\end{equation}
We may also check the validity of Lemma \ref{lemma image of varphi is hyperplane} by noting that
$$ \dim_K (\delta(\mathbb{F})) = p-1 = \dim_K(\mathbb{F}) - 1. $$
Now, taking (\ref{eq example rational functions}) into account and setting $ \sigma = {\rm Id} $, Corollary \ref{cor multiplicity over derivations} says that $ \mathbf{a} = (a_1, a_2, \ldots, a_r) \in \mathbb{F}^r $ is a $ (\sigma,\delta) $-multiplicity sequence if, and only if, there exists
$$ \beta_i = \sum_{j=0}^{p-1} f_{i,j}(z^p) z^j \in k(z), $$
where $ f_{i,j}(z^p) \in k(z^p) $, for $ j =0,1, \ldots, p-1 $, $ f_{i,p-1}(z^p) \neq 0 $ (that is, $ \beta_i \notin \delta(\mathbb{F}) $), and
$$ a_{i+1} = a_i + \delta(\beta_i) \beta_i^{-1}, $$
for $ i = 1,2, \ldots, r-1 $. Fix now an index $ i = 1,2, \ldots, r-1 $. It holds that
$$ \delta(\beta_i) = \sum_{j=1}^{p-1} f_{i,j}(z^p) j \delta(z) z^{j-1} \neq 0, $$
since $ f_{i,p-1}(z^p) (p-1) \delta(z) \in K^* $ and $ 1, z, \ldots, z^{p-2} $ are linearly independent over $ K $. Therefore, we conclude that for all $ a \in \mathbb{F} $ and all $ r \in \mathbb{N} $ such that $ r \geq 2 $, it holds that $ \mathbf{a} = (a,a,\ldots, a) \in \mathbb{F}^r $ is not a $ (\sigma,\delta) $-multiplicity sequence, which is a particularly pathological case. 

As in previous examples, we may directly verify that $ (x-a)^2 \in \mathbb{F}[x;\sigma,\delta] $ has a zero $ b \in \mathbb{F} $ such that $ b \neq a $, in this case for all $ a \in \mathbb{F} $. A direct computation shows that this condition holds if, and only if, there exists $ c \in \mathbb{F}^* $ such that
$$ c^2 = - \delta(c), $$
which indeed does not depend on $ a \in \mathbb{F} $. If we set $ \delta(z) = 1 $ for simplicity (that is, $ \delta = d/dz $ is the usual derivation), then we are done by taking the element $ c = z^{-1} $.
\end{example}

{\footnotesize


}

\end{document}